\documentclass{amsart}
\usepackage{amssymb, amsmath, amsthm, graphics, comment, xspace, enumerate}
\newtheorem{thm}{Theorem}[section]
\newtheorem{cor}[thm]{Corollary}
\newtheorem{lem}[thm]{Lemma}
\newtheorem{prop}[thm]{Proposition}
\theoremstyle{definition}
\newtheorem{defn}[thm]{Definition}
\newtheorem{example}[thm]{Example}
\theoremstyle{remark}
\newtheorem{rem}[thm]{Remark}
\numberwithin{equation}{section}
\begin{document}
\title[Asymptotically Weyl almost periodic functions...]{Asymptotically Weyl almost periodic functions in Lebesgue spaces with variable exponents}

\author{Marko Kosti\' c}
\address{Faculty of Technical Sciences,
University of Novi Sad,
Trg D. Obradovi\' ca 6, 21125 Novi Sad, Serbia}
\email{marco.s@verat.net}

{\renewcommand{\thefootnote}{} \footnote{
\\ \text{  }  \ \    {\it Key words and phrases.} Asymptotically Weyl almost periodic functions with variable exponents, Weyl ergodic components with variable exponents,
abstract degenerate fractional differential inclusions.
\\  \text{  }  \ \ The author is partially supported by grant 174024 of Ministry
of Science and Technological Development, Republic of Serbia.}}

\begin{abstract}
In this paper, we introduce and analyze several different notions of 
Weyl almost periodic functions and Weyl ergodic components in Lebesgue spaces with variable exponent $L^{p(x)}.$ 
We investigate the invariance of  (asymptotical) Weyl almost periodicity with variable exponent under the actions of convolution products, providing also some illustrative applications to abstract fractional differential inclusions in Banach spaces. 
The introduced classes of generalized (asymptotically) Weyl almost periodic functions are new even in the case that the function $p(x)$ has  a constant value $p\geq 1,$ provided that the functions $\phi(x)$ and $F(l,t)$ under our consideration satisfy $\phi(x)\neq x$ or $F(l,t)\neq l^{(-1)/p}.$ 
\end{abstract}
\maketitle

\section{Introduction and preliminaries}

In our joint papers with T. Diagana \cite{AP}-\cite{toka-marek-prim}, we have recently introduced and analyzed several important classes of (asymptotically) Stepanov almost periodic functions and 
(asymptotically) Stepanov almost automorphic functions
in the Lebesgue spaces with variable exponents (see also the earlier papers \cite{m-zitane}-\cite{m-zitane-prim} by T. Diagana and M. Zitane). 
The main purpose of paper \cite{AP} was to consider generalized almost periodicity that intermediate Stepanov and Bohr concept. On the other hand,
the classes introduced by H. Weyl \cite{weyl-initial} and A. S. Kovanko \cite{kovanko} are enormously larger compared with the class of Stepanov almost periodic functions and 
the main purpose of papers \cite{irkutsk}-\cite{brazil} has been to initiate the study of generalized (asymptotical) almost periodicity that intermediate Stepanov and Weyl concept.
In these papers, we have introduced the class of Stepanov
$p$-quasi-asymptotically almost periodic functions  and proved that this class
contains all asymptotically Stepanov $p$-almost periodic functions and makes a subclass
of the class consisting of all Weyl $p$-almost periodic functions ($p\in [1,\infty)$), taken in the
sense of Kovanko's approach \cite{kovanko}.
The main aim of this paper is to continue the research studies raised in \cite{fedorov-novi} and \cite{weyl}-\cite{brazil} by investigating several various classes of asymptotically Weyl  almost periodic functions in Lebesgue spaces with variable exponents
$L^{p(x)}$. 

The organization and main ideas of paper can be briefly described as follows. After introducing the basic concepts, in Subsection \ref{karambita} we recall the main definitions and results about Lebesgue spaces with variable exponents
$L^{p(x)}.$ In Definition \ref{w-orlicz}-Definition \ref{w-orliczaq2}, we introduce the classes of (equi-)Weyl-$(p,\phi,F)$-almost periodic functions and (equi-)Weyl-$(p,\phi,F)_{i}$-almost periodic functions, where $i=1,2.$ The main aim of Proposition \ref{stevate} is to clarify some inclusions between these spaces provided that the function $\phi(\cdot)$ is convex and satisfies certain extra conditions.
In order to ensure the translation invariance of generalized Weyl almost periodic functions with variable exponent, in Definition \ref{w-orliczti}-Definition \ref{w-orlicztistr} we introduce the classes of
(equi-)Weyl-$[p,\phi,F]$-almost periodic functions and (equi-)Weyl-$[p,\phi,F]_{i}$-almost periodic functions, where $i=1,2.$
Several useful comments about these spaces have been provided in Remark \ref{zajev}. In Example \ref{wert}-Example \ref{wert1}, we focus our attention on the following special case: $p(x)\equiv p \in [1,\infty),$ $\phi(x)=x$ and $F(l,t)=l^{(-1)/p\sigma},$ $\sigma \in {\mathbb R},$
which is the most important for investigation of generalized almost periodicity that intermediates Stepanov and Weyl concept. In Section \ref{weylerg}, we introduce and analyze various types of Weyl ergodic components with variable exponent
and asymptotically Weyl almost periodic functions with variable exponent. As mentioned in the abstract, the introduced classes of generalized (asymptotically) Weyl almost periodic functions are new even in the case that the function $p(x)$ has  a constant value $p\geq 1$ and $\phi(x)\neq x$ or $F(l,t)\neq l^{(-1)/p(t)}.$  From the application point of view, the main results of paper are given in Section \ref{cp-inv}, where we examine the invariance of generalized Weyl almost periodicity with variable exponent under the action of convolution products and the convolution invariance of Weyl almost periodic functions with variable exponent. In order to do that, we shall basically follow the method proposed in the proof of Theorem \ref{jensen}.
Section 4 contains two subsections; in Subsection \ref{rates}, we consider the case in which the exponent $p(x)\equiv p\in [1,\infty)$ is constant and solution operator family $(R(t))_{t>0}\subseteq L(X,Y)$ has a certain growth order around the points zero, plus infinity and, in Subsection \ref{nekiprimeri}, we provide some illustrative applications in the qualitative analysis of solutions to abstract degenerate fractional differential equations with Weyl-Liouville or Caputo derivatives.
The paper is not exhaustively complete and we feel it is our duty to say that
the classes of Weyl quasi-asymptotically almost periodic functions, Besicovitch quasi-asymptotically almost periodic functions with constant (variable) exponent and corresponding spaces of generalized almost automorphic functions will be considered 
somewhere else (cf. also the research monographs \cite{diagana} by T. Diagana and \cite{gaston} by G. M. N'Gu\' er\' ekata). Two-parameter asymptotically Weyl almost periodic functions with variable exponents and related composition principles will be considered 
somewhere else, as well (cf. the paper \cite{mellah} by F. Bedouhene, Y. Ibaouene, O. Mellah, P. Raynaud de Fitte and the research monograph \cite{nova-mono} by the author for further information about the subject).

We use the standard notation throughout the paper.
By $(X,\| \cdot \|)$ and $(Y, \|\cdot\|_Y)$ we denote two complex Banach spaces. Further on, by
$L(X,Y)$ we denote the Banach space consisting of all bounded linear operators from $X$ into
$Y;$ $L(X)\equiv L(X,X).$ %If $A: D(A) \subseteq X \mapsto X$ is a linear operator,
%then its nullspace (or kernel) and range will be denoted respectively by
%$N(A)$ and $R(A)$. %Further, we will identify $A$ with its graph defined by $\{(x, Ax): x \in X\}$. %By $[D(A)]$ we denote the Banach space 
%$(D(A), \|\cdot\|_{[D(A)]}),$ where $\|\cdot\|_{[D(A)]}$ is the graph norm defined by $\|x\|_{[D(A)]}:=\|x\|+\|Ax\|$ for all $x\in D(A).$

For given constants $s\in {\mathbb R}$ and $\theta \in (0,\pi]$, we set $\lfloor s \rfloor :=\sup \{
l\in {\mathbb Z} : s\geq l \}$, $\lceil s \rceil:=\inf \{
l\in {\mathbb Z} : s\leq l \}$ and $\Sigma_{\theta}:=\{ z\in {\mathbb C} \setminus \{0\} :
|\arg (z)|<\theta \}.$ 
The Euler Gamma function is denoted by
$\Gamma(\cdot)$. We also set $g_{\zeta}(t):=t^{\zeta-1}/\Gamma(\zeta),$ $\zeta>0.$
The finite convolution operator $\ast$ is defined by $f\ast g(t):=\int_{0}^{t}f(t-s)g(s)\,
ds;$ since there is no risk for confusion, the infinite convolution product 
$t\mapsto \int_{-\infty}^{+\infty}f(t-s)g(s)\,
ds,$ $t\in {\mathbb R}$ will be denoted by the same symbol.
For any function $g : {\mathbb R} \rightarrow X ,$ we define the function $\check{g} : {\mathbb R} \rightarrow X$ by $\check{g}(t):=g(-t),$ $t\in {\mathbb R}.$ If $f : I\rightarrow X$, where $I=[0,\infty)$ or $I={\mathbb R},$ then we define the Bochner transform 
$
\hat{f} : I \rightarrow L^{p}([0,1] :X)
$ of function $f(\cdot)$ by
$
\hat{f}(t)(s):=f(t+s),$ $t\in I,$ $ s\in [0,1].
$

Suppose that $f: [a,b]\rightarrow  \mathbb{R}$ is a non-negative Lebesgue-integrable function, where $a,\ b \in {\mathbb R},$ $a<b,$ and $\phi :[0,\infty) \rightarrow {\mathbb R}$ is a convex function. Let us recall that the Jensen integral inequality states that 
$$
{\displaystyle \phi \left({\frac {1}{b-a}}\int _{a}^{b}f(x)\,dx\right)\leq {\frac {1}{b-a}}\int _{a}^{b}\phi (f(x))\,dx.}
$$
Using this integral inequality, we can simply prove that, for every two sequences $(a_{k})$ and $(x_{k})$ of non-negative real numbers such that $\sum_{k=0}^{\infty}a_{k}=1,$ we have
\begin{align}\label{infinitever}
\phi \Biggl( \sum_{k=0}^{\infty}a_{k}x_{k} \Biggr) \leq \sum_{k=0}^{\infty}a_{k}\phi \bigl(x_{k}\bigr).
\end{align}
If $\phi : [0,\infty) \rightarrow {\mathbb R}$ is a concave function, then the above inequalities reverse.

Suppose that $X$ and $Y$ are two Banach spaces.
A multivalued map (multimap) ${\mathcal A} : X \rightarrow P(Y)$ is said to be a multivalued
linear operator, MLO for short, iff the following holds:
\begin{itemize}
\item[(i)] $D({\mathcal A}) := \{x \in X : {\mathcal A}x \neq \emptyset\}$ is a linear subspace of $X$;
\item[(ii)] ${\mathcal A}x +{\mathcal A}y \subseteq {\mathcal A}(x + y),$ $x,\ y \in D({\mathcal A})$
and $\lambda {\mathcal A}x \subseteq {\mathcal A}(\lambda x),$ $\lambda \in {\mathbb C},$ $x \in D({\mathcal A}).$
\end{itemize}
In the case that $X=Y,$ then we say that ${\mathcal A}$ is an MLO in $X.$ For more details about multivalued
linear operators and degenerate resolvent operator families, the reader may consult \cite{faviniyagi}, \cite{FKP} and references cited therein.

By $C_{0}([0,\infty) : X)$ we denote the space consisting of all continuous functions $f : [0,\infty) \rightarrow X$ such that $\lim_{t\rightarrow +\infty}f(t)=0.$ Equipped with the sup-norm, 
$C_{0}([0,\infty) : X)$
becomes a Banach space. By $c_{0}$ we denote the Banach space of all numerical sequences tending to zero, equipped with the sup-norm.

\subsection{Lebesgue spaces with variable exponents
$L^{p(x)}$}\label{karambita}

The monograph \cite{variable} by L. Diening, P. Harjulehto, P. H\"ast\"uso and M. Ruzicka
is of invaluable importance in the study of Lebesgue spaces with variable exponents.

Let $\emptyset \neq \Omega \subseteq {\mathbb R}$ be a nonempty subset and let 
$M(\Omega  : X)$ stand for the collection of all measurable functions $f: \Omega \rightarrow X;$ $M(\Omega):=M(\Omega : {\mathbb R}).$ Furthermore, ${\mathcal P}(\Omega)$ denotes the vector space of all Lebesgue measurable functions $p : \Omega \rightarrow [1,\infty].$
For any $p\in {\mathcal P}(\Omega)$ and $f\in M(\Omega : X),$ set
$$
\varphi_{p(x)}(t):=\left\{
\begin{array}{l}
t^{p(x)},\quad t\geq 0,\ \ 1\leq p(x)<\infty,\\ \\
0,\quad 0\leq t\leq 1,\ \ p(x)=\infty,\\ \\
\infty,\quad t>1,\ \ p(x)=\infty 
\end{array}
\right.
$$
and
\begin{align}\label{modular}
\rho(f):=\int_{\Omega}\varphi_{p(x)}(\|f(x)\|)\, dx .
\end{align}
We define the Lebesgue space 
$L^{p(x)}(\Omega : X)$ with variable exponent
by
$$
L^{p(x)}(\Omega : X):=\Bigl\{f\in M(\Omega : X): \lim_{\lambda \rightarrow 0+}\rho(\lambda f)=0\Bigr\}.
$$
Equivalently,
\begin{align*}
L^{p(x)}(\Omega : X)=\Bigl\{f\in M(\Omega : X):  \mbox{ there exists }\lambda>0\mbox{ such that }\rho(\lambda f)<\infty\Bigr\};
\end{align*}
see, e.g., \cite[p. 73]{variable}.
For every $u\in L^{p(x)}(\Omega : X),$ we introduce the Luxemburg norm of $u(\cdot)$ in the following manner:
$$
\|u\|_{p(x)}:=\|u\|_{L^{p(x)}(\Omega :X)}:=\inf\Bigl\{ \lambda>0 : \rho(f/\lambda)    \leq 1\Bigr\}.
$$
Equipped with the above norm, the space $
L^{p(x)}(\Omega : X)$ becomes a Banach space (see, e.g., \cite[Theorem 3.2.7]{variable} for the scalar-valued case), coinciding with the usual Lebesgue space $L^{p}(\Omega : X)$ in the case that $p(x)=p\geq 1$ is a constant function.
For any $p\in M(\Omega),$ we set 
$$
p^{-}:=\text{essinf}_{x\in \Omega}p(x) \ \ \mbox{ and } \ \ p^{+}:=\text{esssup}_{x\in \Omega}p(x).
$$
Define
$$
C_{+}(\Omega ):=\bigl\{ p\in M(\Omega): 1<p^{-}\leq p(x) \leq p^{+} <\infty \mbox{ for a.e. }x\in \Omega \bigr \}
$$
and
$$
D_{+}(\Omega ):=\bigl\{ p\in M(\Omega): 1 \leq p^{-}\leq p(x) \leq p^{+} <\infty \mbox{ for a.e. }x\in \Omega \bigr \}.
$$
For $p\in D_{+}(\Omega),$ the space $
L^{p(x)}(\Omega : X)$ behaves nicely, with almost all fundamental properties of the Lesbesgue space with constant exponent $
L^{p}(\Omega : X)$ being retained; in this case, we know that the function $\rho(\cdot)$ given by \eqref{modular} is modular in the sense of \cite[Definition 2.1.1]{variable}, as well as that
$$
L^{p(x)}(\Omega : X)=\Bigl\{f\in M(\Omega : X):  \mbox{ for all }\lambda>0\mbox{ we have }\rho(\lambda f)<\infty\Bigr\}.
$$
Furthermore, if $p\in 
C_{+}(\Omega ),$ then $
L^{p(x)}(\Omega : X)$ is uniformly convex and thus reflexive (\cite{fan-zhao}). 

We will use the following lemma (see, e.g., \cite[Lemma 3.2.20, (3.2.22); Corollary 3.3.4; p. 77]{variable} for the scalar-valued case):

\begin{lem}\label{aux}
\begin{itemize}
\item[(i)] (The  H\"older inequality) Let $p,\ q,\ r \in {\mathcal P}(\Omega)$ such that
$$
\frac{1}{q(x)}=\frac{1}{p(x)}+\frac{1}{r(x)},\quad x\in \Omega .
$$
Then, for every $u\in L^{p(x)}(\Omega : X)$ and $v\in L^{r(x)}(\Omega),$ we have $uv\in L^{q(x)}(\Omega : X)$
and
\begin{align*}
\|uv\|_{q(x)}\leq 2 \|u\|_{p(x)}\|v\|_{r(x)}.
\end{align*}
\item[(ii)] Let $\Omega $ be of a finite Lebesgue's measure and let $p,\ q \in {\mathcal P}(\Omega)$ such $q\leq p$ a.e. on $\Omega.$ Then
 $L^{p(x)}(\Omega : X)$ is continuously embedded in $L^{q(x)}(\Omega : X).$
\item[(iii)] Let $f\in L^{p(x)}(\Omega : X),$ $g\in M(\Omega : X)$ and $0\leq \|g\| \leq \|f\|$ a.e. on $\Omega .$ Then $g\in L^{p(x)}(\Omega : X)$ and $\|g\|_{p(x)}\leq \|f\|_{p(x)}.$
\end{itemize}
\end{lem}

For additional details upon Lebesgue spaces with variable exponents
$L^{p(x)},$ we refer the reader to the following sources: \cite{m-zitane}, \cite{m-zitane-prim}, \cite{fan-zhao} and \cite{doktor}.

Before proceeding further, we need to recall the recently introduced notions of $S^{p(x)}$-boundedness and (asymptotical) Stepanov $p(x)$-almost periodicity:

\begin{defn}\label{daniel-toka} (\cite{AP})
Let $p\in {\mathcal P}([0,1])$ and let $I={\mathbb R}$ or $I=[0,\infty).$ A function $f\in M(I : X)$ is said to be Stepanov $p(x)$-bounded (or $S^{p(x)}$-bounded) iff $f(\cdot +t) \in L^{p(x)}([0,1]: X)$ for all $t\in I,$ and the sup norm of Bochner transform satisfies $\sup_{t\in I} \|f(\cdot +t)\|_{p(x)}<\infty ;$ more precisely, 
$$
\|f\|_{S^{p(x)}}:=\sup_{t\in I}\inf\Biggl\{ \lambda>0 : \int_{0}^{1}\varphi_{p(x)}\Biggl( \frac{\|f(x +t)\|}{\lambda}\Biggr)\, dx \leq 1\Biggr\}<\infty.
$$
The collection of such functions will be denoted by $L_{S}^{p(x)}(I:X).$ 
\end{defn}

From Definition \ref{daniel-toka} it follows that the space $L_{S}^{p(x)}(I:X)$ is translation invariant in the sense that, for every $f\in L_{S}^{p(x)}(I:X)$ and $\tau \in I,$ we have $f(\cdot+\tau) \in L_{S}^{p(x)}(I:X).$ This is not the case with the notion introduced by T. Diagana and M. Zitane in \cite{m-zitane}-\cite{m-zitane-prim}.

\begin{defn}\label{sasasa} (\cite{AP})
\begin{itemize}
\item[(i)]
Let $p\in {\mathcal P}([0,1])$ and let $I={\mathbb R}$ or $I=[0,\infty).$ A function $f\in L_{S}^{p(x)}(I:X)$ is said to be Stepanov $p(x)$-almost periodic (or Stepanov $p(x)$-a.p.) iff the function $\hat{f} : I \rightarrow L^{p(x)}([0,1]: X)$ is almost periodic. 
The collection of such functions will be denoted by $APS^{p(x)}(I : X)$.
\item[(ii)] Let $p\in {\mathcal P}([0,1])$ and let $I=[0,\infty).$ A function $f\in L_{S}^{p(x)}(I:X)$ is said to be asymptotically Stepanov $p(x)$-almost periodic (or asymptotically Stepanov $p(x)$-a.p.) iff the function $\hat{f} : I \rightarrow L^{p(x)}([0,1]: X)$ is  asymptotically almost periodic. The collection of such functions will be denoted by $AAPS^{p(x)}(I : X)$. 
\end{itemize}
\end{defn}

We will extend \cite[Definition 3.10]{m-zitane} in the following way (in this paper, the authors have considered the case $I={\mathbb R}$ and $p\in C_{+}({\mathbb R});$ we can extend the notion introduced in \cite[Definition 3.11]{m-zitane} in the same way):

\begin{defn}\label{defrinoli}
Let $I={\mathbb R}$ or $I=[0,\infty),$ and let $p\in {\mathcal P}(I).$ Then it is said that a measurable function $ f : I \rightarrow X$ belongs to the space $BS^{p(x)}(I : X)$ iff
$$
\bigl \|f \bigr \|_{{\bf S}^{p(x)}}:=\sup_{t\in I} \inf \Biggl\{ \lambda>0 : \int^{t+1}_{t}\varphi_{p(x)}(\|f(x)\|/\lambda)\, dx \leq 1 \Biggr\}<\infty.
$$
\end{defn}

Before starting our work, the author would like to express his sincere gratitude to Professor F. Boulahia for many stimulating discussions and useful suggestions during the research.

\section{$(p,\phi,F)$-Classes and $[p,\phi,F]$-classes of Weyl almost periodic functions}\label{fi-prcko}

Throughout this paper, we assume the following general conditions:
\begin{itemize}
\item[(A):] $I={\mathbb R}$ or $I=[0,\infty),$ $\phi :[0,\infty) \rightarrow [0,\infty) ,$ $ p\in {\mathcal P}(I)$ and 
$F :(0,\infty) \times I \rightarrow (0,\infty).$
\item[(B):] The same as (A) with the assumption $ p\in {\mathcal P}(I)$ replaced by $ p\in {\mathcal P}([0,1])$ therein.
\end{itemize}

We introduce the notions of an (equi-)Weyl-$(p,\phi,F)$-almost periodic function and an (equi-)Weyl-$(p,\phi,F )_{i}$-almost periodic function, where $i=1,2,$ as follows (see \cite{nova-mono} for the case that $p(x)\equiv p\in [1,\infty),$ $\phi(x)=x$ and $F(l,t)=l^{(-1)/p}$, when we deal with the usually considered (equi-)Weyl-$p$-almost periodic functions, as well as to \cite[Remark 4.13]{AP} for the case  that $\phi(x)=x$ and and $F(l,t)=l^{(-1)/p(t)}$): 

\begin{defn}\label{w-orlicz}
Suppose that condition \mbox{(A)} holds, $f : I \rightarrow X$ and $\phi(\|f(\cdot +\tau)-f(\cdot)\|) \in L^{p(x)}( K)$ for any $\tau \in I$ and any compact subset $K$ of $I.$ 
\begin{itemize}
\item[(i)] It is said that the function $f(\cdot)$ is equi-Weyl-$(p,\phi,F)$-almost periodic, $f\in e-W_{ap}^{(p,\phi,F)}(I:X)$ for short, iff for each $\epsilon>0$ we can find two real numbers $l>0$ and $L>0$ such that any interval $I'\subseteq I$ of length $L$ contains a point $\tau \in  I'$ such that
\begin{align}\label{profice}
e-\|f\|_{(p,\phi,F,\tau)}:=\sup_{t\in I}\Biggl[ F(l,t) \Bigl[\phi \bigl( \bigl\| f(\cdot+\tau) -f(\cdot)\bigr\|\bigr)_{L^{p(\cdot)}[t,t+l]}\Bigr]\Biggr] \leq \epsilon .
\end{align}
\item[(ii)] It is said that the function $f(\cdot)$ is Weyl-$(p,\phi,F)$-almost periodic, $f\in W_{ap}^{(p,\phi,F)}(I: X)$ for short, iff for each $\epsilon>0$ we can find a real number $L>0$ such that any interval $I'\subseteq I$ of length $L$ contains a point $\tau \in  I'$ such that
\begin{align}\label{profice1}
\|f\|_{(p,\phi,F,\tau)}:=\limsup_{l\rightarrow \infty} \sup_{t\in I}\Biggl[F(l,t) \Bigl[\phi \bigl(\bigl \| f(\cdot+\tau) -f(\cdot)\bigr\|\bigr)_{L^{p(\cdot)}[t,t+l]} \Bigr]\Biggr]  \leq \epsilon .
\end{align}
\end{itemize}
\end{defn}

\begin{defn}\label{w-orliczaq}
Suppose that condition \mbox{(A)} holds, $f : I \rightarrow X$ and $\|f(\cdot +\tau)-f(\cdot)\| \in L^{p(x)}( K)$ for any $\tau \in I$ and any compact subset $K$ of $I.$ 
\begin{itemize}
\item[(i)] It is said that the function $f(\cdot)$ is equi-Weyl-$( p,\phi,F )_{1}$-almost periodic, $f\in e-W_{ap}^{( p,\phi,F)_{1}}(I:X)$ for short, iff for each $\epsilon>0$ we can find two real numbers $l>0$ and $L>0$ such that any interval $I'\subseteq I$ of length $L$ contains a point $\tau \in  I'$ such that
\begin{align*} 
e-\|f\|_{(p,\phi,F,\tau)_{1}}:=\sup_{t\in I}\Biggl[ F(l,t) \phi \Bigl[\bigl( \bigl\| f(\cdot+\tau) -f(\cdot)\bigr\|\bigr)_{L^{p(\cdot)}[t,t+l]}\Bigr]\Biggr] \leq \epsilon .
\end{align*}
\item[(ii)] It is said that the function $f(\cdot)$ is Weyl-$(p,\phi,F)_{1}$-almost periodic, $f\in W_{ap}^{(p,\phi,F)_{1}}(I: X)$ for short, iff for each $\epsilon>0$ we can find a real number $L>0$ such that any interval $I'\subseteq I$ of length $L$ contains a point $\tau \in  I'$ such that
\begin{align*}
\|f\|_{(p,\phi,F,\tau)_{1}}:=\limsup_{l\rightarrow \infty} \sup_{t\in I}\Biggl[F(l,t)  \phi \Bigl[\bigl(\bigl \| f(\cdot+\tau) -f(\cdot)\bigr\|\bigr)_{L^{p(\cdot)}[t,t+l]} \Bigr]\Biggr]  \leq \epsilon .
\end{align*}
\end{itemize}
\end{defn}

\begin{defn}\label{w-orliczaq2}
Suppose that condition \mbox{(A)} holds, $f : I \rightarrow X$ and $\|f(\cdot +\tau)-f(\cdot)\| \in L^{p(x)}( K)$ for any $\tau \in I$ and any compact subset $K$ of $I.$ 
\begin{itemize}
\item[(i)] It is said that the function $f(\cdot)$ is equi-Weyl-$( p,\phi,F )_{2}$-almost periodic, $f\in e-W_{ap}^{( p,\phi,F)_{2}}(I:X)$ for short, iff for each $\epsilon>0$ we can find two real numbers $l>0$ and $L>0$ such that any interval $I'\subseteq I$ of length $L$ contains a point $\tau \in  I'$ such that
\begin{align*}
e-\|f\|_{(p,\phi,F,\tau)_{2}}:=\sup_{t\in I} \phi \Biggl[ F(l,t) \Bigl[\bigl( \bigl\| f(\cdot+\tau) -f(\cdot)\bigr\|\bigr)_{L^{p(\cdot)}[t,t+l]}\Bigr]\Biggr] \leq \epsilon .
\end{align*}
\item[(ii)] It is said that the function $f(\cdot)$ is Weyl-$(p,\phi,F)_{2}$-almost periodic, $f\in W_{ap}^{(p,\phi,F)_{2}}(I: X)$ for short, iff for each $\epsilon>0$ we can find a real number $L>0$ such that any interval $I'\subseteq I$ of length $L$ contains a point $\tau \in  I'$ such that
\begin{align*}
\|f\|_{(p,\phi,F,\tau)_{2}}:=\limsup_{l\rightarrow \infty}\sup_{t\in I} \phi \Biggl[F(l,t)  \Bigl[\bigl(\bigl \| f(\cdot+\tau) -f(\cdot)\bigr\|\bigr)_{L^{p(\cdot)}[t,t+l]} \Bigr]\Biggr]  \leq \epsilon .
\end{align*}
\end{itemize}
\end{defn}

If $i=1,2$ and $F(l,t)=\psi(l)^{(-1)/p(t)}$ for some function $\psi : (0,\infty) \rightarrow (0,\infty)$ and all $t\in I,$ then we also say that the function $f(\cdot)$ is (equi-)Weyl-$(p,\phi,\psi)$-almost periodic, resp.  (equi-)Weyl-$(p,\phi,\psi )_{i}$-almost periodic, when the corresponding class of functions is also denoted by $(e-)W_{ap}^{(p,\phi,\psi)}(I:X),$ resp. 
$(e-)W_{ap}^{( p,\phi,\psi )_{i}} (I:X).$ There is no need to say that the above classes coincide in the case that $\phi(x)\equiv x.$

\begin{example}\label{primer-triv}
\begin{itemize}
\item[(i)] If $\phi(0)=0,$ then any continuous periodic function $f : I \rightarrow X$ belongs to any of the above introduced function spaces. If $\phi(0)>0,$ then 
a constant function cannot belong to any of the function spaces introduced in Definition \ref{w-orliczaq2}, while the function spaces introduced in Definition \ref{w-orlicz}-Definition \ref{w-orliczaq} can contain constant functions (see also Remark \ref{zajev1}(iii)). 
\item[(ii)] If $\phi(x)=x$ and $p(x)\equiv p \in [1,\infty),$ then any Stepanov $p$-bounded function $f : I \rightarrow X$ belongs to any of the above introduced function spaces with $F(l,t)\equiv l^{-\sigma},$ where $\sigma >1/p;$ in particular, if $f(\cdot)$ is Stepanov $p(x)$-bounded and $p\in D_{+}(I),$ then $f(\cdot)$ belongs to any of the above introduced function spaces with $F(l,t)\equiv l^{-\sigma},$ where $\sigma >1/p^{+}.$
This simply follows from the inequality
$$
\Biggl(\int^{t+l}_{t}\|f(s+\tau)-f(s)\|^{p}\, ds\Biggr)^{1/p}\leq \sum_{k=0}^{\lfloor l \rfloor}\Biggl(\int^{t+k+1}_{t+k}\|f(s+\tau)-f(s)\|^{p}\, ds\Biggr)^{1/p},
$$
which is valid for any $t,\ \tau \in I,$ $l>0,$ 
and a simple argumentation. Suppose now that $I={\mathbb R}$ or $I=[0,\infty),$ $p\in {\mathcal P}(I)$ and $f\in BS^{p(x)}(I : X).$ A similar line of reasoning shows that
$f(\cdot)$ belongs to any of the above introduced function spaces provided that
\begin{itemize}
\item[(a)] $p\in D_{+}(I)$ and
$F(l,t)\equiv l^{-\sigma},$ where $\sigma >1/p^{+},$ or
\item[(b)] $F(l,t)\equiv l^{-\sigma},$ where $\sigma >1,$ in general case. For this, it is only worth noting that we have $\varphi_{p(x)}(t/l^{\sigma})\leq (1/l^{\sigma})\varphi_{p(x)}(t)$ for any $t\geq 0$ and $l\geq 1.$
\end{itemize}
\item[(iii)] If $X$ does not contain an isomorphic copy of the sequence space
$c_{0}$, $\phi(x)=x$ and $F(l,t)\equiv F(t),$ where $\lim_{t\rightarrow +\infty}F(t)=+\infty,$ then there is no trigonometric polynomial $f(\cdot)$ and function $p\in {\mathcal P}({\mathbb R})$ such that $f\in e-W_{ap}^{(p,x,F)}({\mathbb R}:X).$ 
If we
suppose the contrary, then using the fact that the space $L^{p(x)}[t,t+l]$ is continuously embedded into the space $L^{1}[t,t+l]$ with the constant of embeddings less than or equal to $2(1+l)$ (see, e.g., \cite[Corollary 3.3.4]{variable}), where $t\in {\mathbb R}$ and $l>0,$ we get that
for each $\epsilon>0$ we can find two real numbers $l>0$ and $L>0$ such that any interval $I'\subseteq {\mathbb R}$ of length $L$ contains a point $\tau \in  I'$ such that
\begin{align}\label{koeradenija}
\sup_{t\in {\mathbb R}}\Biggl[ F(t) \bigl\| f(\cdot+\tau) -f(\cdot)\bigr\|_{L^{1}[t,t+l]}\Biggr] \leq 2\epsilon (1+l) .
\end{align}
Let such numbers $l>0$ and $\tau \in {\mathbb R}$ be fixed.
By \eqref{koeradenija}, we get that the mapping $t\mapsto f_{1}(t)\equiv \int^{t+l}_{t}\|f(s+\tau)-f(s)\|\, ds,$ $t\geq 0$ belongs to the space $C_{0}([0,\infty) : {\mathbb C}).$
On the other hand, the mapping $s\mapsto \|f(s+\tau)-f(s)\|,$ $s\in {\mathbb R}$ is almost periodic and satisfies that $\int^{t}_{0} \|f(s+\tau)-f(s)\|\, ds <\infty,$ so that the mapping $t\mapsto f_{2}(t)\equiv \int^{t}_{0} \|f(s+\tau)-f(s)\|\, ds,$ $t\in {\mathbb R}$ is almost periodic
by \cite[Theorem 2.1.1(vi)]{nova-mono}.
By the translation invariance, the same holds for the mapping $f_{1}(\cdot)=f_{2}(\cdot +\tau)-f_{2}(\cdot).$ Since $f_{1}\in C_{0}([0,\infty) : {\mathbb C}),$ we get that $f_{1}\equiv 0,$ so that
$\|f(s+\tau)-f(s)\|=0$ for all $s\geq 0$ and $f(\cdot)$ is periodic, which is a contradiction. Based on the conclusion obtained in this part, we will not examine the question whether, for a given number $\epsilon>0$ and an equi-Weyl-$( p,\phi,F )$-almost periodic function or an equi-Weyl-$( p,\phi,F )_{i}$-almost periodic
function ($i=1,2$), we can find a trigonometric polynomial $P(\cdot)$ such that $\|P-f\|_{(p,\phi,F)}<\epsilon$ or $\|P-f\|_{(p,\phi,F)_{i}}<\epsilon$  ($i=1,2$), where
$$
e-\|f\|_{(p,\phi,F)}:=\sup_{t\in I}\Biggl[ F(l,t) \Bigl[\phi \bigl( \bigl\| f(\cdot)\bigr\|\bigr)_{L^{p(\cdot)}[t,t+l]}\Bigr]\Biggr],
$$
$$
e-\|f\|_{(p,\phi,F)_{1}}:=\sup_{t\in I}\Biggl[ F(l,t) \phi \Bigl[\bigl( \bigl\| f(\cdot)\bigr\|\bigr)_{L^{p(\cdot)}[t,t+l]}\Bigr]\Biggr]
$$
and
$$
e-\|f\|_{(p,\phi,F)_{2}}:=\sup_{t\in I} \phi \Biggl[ F(l,t) \Bigl[\bigl( \bigl\| f(\cdot)\bigr\|\bigr)_{L^{p(\cdot)}[t,t+l]}\Bigr]\Biggr] .
$$ 
For the usually considered class of equi-Weyl-$p$-almost periodic functions, where $1\leq p<\infty,$ the answer to the above question is affirmative (see, e.g., \cite[Theorem 2.3.2]{nova-mono}). Observe also that the sub-additivity of function $\phi(\cdot)$ implies the
sub-additivity of 
functions $
e-\|\cdot\|_{(p,\phi,F)}$ and $
e-\|\cdot\|_{(p,\phi,F)_{i}},$ where $i=1,2;$ since the limit superior is also a sub-additive operation, the same holds for the functions $
\|\cdot\|_{(p,\phi,F)}$ and $
\|\cdot\|_{(p,\phi,F)_{i}},$ where $i=1,2,$ defined as above (cf. the second parts of Definition \ref{w-orlicz}-Definition \ref{w-orliczaq2}, as well as Definition \ref{w-orliczti}-Definition \ref{w-orlicztistr} below).
\end{itemize}
\end{example}

In the case that the function $\phi(\cdot)$ is convex and $p(x)\equiv 1$, we have the following result:

\begin{prop}\label{stevate}
Suppose that $p(x)\equiv 1,$ $f : I \rightarrow X,$ $\|f(\cdot +\tau)-f(\cdot)\| \in L^{p(x)}( K)$ for any $\tau \in I$ and any compact subset $K$ of $I,$ as well as condition
\begin{itemize}
\item[(C):] $\phi(\cdot)$ is convex and there exists a function $\varphi :[0,\infty) \rightarrow [0,\infty)$ such that $\phi(lx)\leq \varphi(l) \phi(x)$ for all $l>0$ and $x\geq 0$  
\end{itemize}
holds. Set $F_{1}(l,t):=F(l,t)l[\varphi(l)]^{-1},$ $l>0$, $t\in I$ and $F_{2}(l,t):=l^{-1}\varphi ( F(l,t)l),$ $l>0$, $t\in I$.
Then we have:
\begin{itemize}
\item[(i)] $f\in (e-)W_{ap}^{(1,\phi,F)} \Rightarrow f\in (e-)W_{ap}^{(1,\phi,F_{1})_{1}}.$ 
\item[(ii)] $f\in (e-)W_{ap}^{(1,\phi, F_{2})} \Rightarrow f\in (e-)W_{ap}^{(1,\phi,F)_{2}}.$ 
\end{itemize}
\end{prop}

\begin{proof}
To prove (i), suppose that $f\in (e-)W_{ap}^{(1,\phi,F)}.$ Then the assumption (C) and the Jensen integral inequality together imply
\begin{align*}
\phi \Bigl(& \|f(\cdot +\tau)-f(\cdot)\|_{L^{1}[t,t+l]}  \Bigr)
= \phi \Bigl( l \cdot l^{-1} \|f(\cdot +\tau)-f(\cdot)\|_{L^{1}[t,t+l]}  \Bigr)
\\ \leq & \varphi(l) \phi \Bigl( l^{-1} \|f(\cdot +\tau)-f(\cdot)\|_{L^{1}[t,t+l]} \Bigr)
 \leq \varphi(l) l^{-1}\Bigl[ \phi \bigl( \|f(\cdot +\tau)-f(\cdot)\|\bigr)\Bigr]_{L^{1}[t,t+l]}.
\end{align*}
This simply yields $f\in (e-)W_{ap}^{(1,\phi,F_{1})_{1}}.$ To prove (ii), suppose that $f\in (e-)W_{ap}^{(1,\phi, F_{2})}.$ Then the assumption (C) and the Jensen integral inequality together imply
\begin{align*}
\phi\Bigl( & F(l,t)\|f(\cdot+\tau)-f(\cdot)\|_{L^{1}[t,t+l]} \Bigr)=\phi\Bigl( F(l,t)l\cdot l^{-1} \|f(\cdot+\tau)-f(\cdot)\|_{L^{1}[t,t+l]} \Bigr)
\\ \leq & \varphi(F(t,l)l)l^{-1}\Bigl[ \phi \bigl( \|f(\cdot+\tau)-f(\cdot)\|\bigr) \Bigr]_{L^{1}[t,t+l]}.
\end{align*}
This simply yields $f\in (e-)W_{ap}^{(1,\phi,F)_{2}}.$ 
\end{proof}

Before we go any further, let us recall that any equi-Weyl-$p$-almost periodic function needs to be Weyl $p$-almost periodic, while the converse statement does not hold in general. On the other hand, it is not true that an equi-Weyl-$(p,\phi,\psi)$-almost periodic function, resp. equi-Weyl-$(p,\phi,\psi)_{i}$-almost periodic function, is Weyl-$(p,\phi,\psi)$-almost periodic, resp. Weyl-$(p,\phi,\psi)_{i}$-almost periodic; moreover, an unrestrictive choice of function $\psi (\cdot)$
allows us to work with a substantially large class of quasi-almost periodic functions: As it can be simply approved, any Stepanov $p$-almost periodic function $f(\cdot)$ 
is equi-Weyl-$(p,\phi,\psi)$-almost periodic with $p(x)\equiv p \in [1,\infty),$ $\psi(l)\equiv 1,$ $\phi(x)=x;$ on the other hand, any continuous Stepanov $p$-almost periodic function $f(\cdot)$ which is not periodic cannot be Weyl-$(p,x,1)$-almost periodic, for example. Let us explain the last fact in more detail. 
If we suppose the contrary, then for each $\epsilon>0$ we can find a real number $L>0$ such that any interval $I'\subseteq I$ of length $L$ contains a point $\tau \in  I'$ such that
\eqref{profice1} holds with $p(x)\equiv p \in [1,\infty),$ $\psi(l)\equiv 1$ and $\phi(x)=\varphi(x)=x.$
This simply implies that for each $\epsilon>0$ we can find a strictly increasing sequence $(l_{n})$ of positive real numbers tending to infinity such that for each $t\in I$ and $n\in {\mathbb N}$ we have $\int_{t
+l_{n}}^{t}\| f(x+\tau)-f(x)\|^{p}\, dx \leq \epsilon$ for each $\epsilon>0;$ hence, $\int_{I}\| f(x+\tau)-f(x)\|^{p}\, dx \leq \epsilon$ and therefore 
$\int_{I}\| f(x+\tau)-f(x)\|^{p}\, dx=0.$ This yields $f(x+\tau)=f(x),$ $x\in I,$ which is a contradiction with our preassumption. 

\begin{rem}\label{zajev1}
\begin{itemize}
\item[(i)] It is clear that, if $f(\cdot)$ is an (equi-)Weyl-$(p,\phi,F)$-almost periodic function, resp. (equi-)Weyl-$(p,\phi,F)_{1}$-almost periodic function, and 
$F(l,t)\geq F_{1}(l,t)$ for every $l>0$ and $t\in I,$ then $f(\cdot)$ is (equi-)Weyl-$(p,\phi,F_{1})$-almost periodic, resp. (equi-)Weyl-$(p,\phi,F_{1})_{1}$-almost periodic. Furthermore,
if $f(\cdot)$ is an (equi-)Weyl-$(p,\phi,F)_{2}$-almost periodic function, then $f(\cdot)$ is an (equi-)Weyl-$(p,\phi,F_{1})_{2}$-almost periodic function 
provided that $F(l,t)\geq F_{1}(l,t)$ for every $l>0,$ $t\in I$ and $\phi(\cdot)$ is monotonically increasing, or $F(l,t)\leq F_{1}(l,t)$ for every $l>0,$ $t\in I$ and $\phi(\cdot)$ is monotonically decreasing.
\item[(ii)] If $f(\cdot)$ is an (equi-)Weyl-$(p,\phi,F)$-almost periodic function, resp.  (equi-)Weyl-$(p,\phi,F)_{i}$-almost periodic function, $\phi_{1}(\cdot)$ is measurable and $0\leq \phi_{1}\leq \phi,$
then Lemma \ref{aux}(iii) yields that $f(\cdot)$ is (equi-)Weyl-$(p,\phi_{1},F)$-almost periodic, resp. (equi-)Weyl-$(p,\phi_{1},F)_{i}$-almost periodic, where $i=1,2.$
\item[(iii)] Regarding the first parts in the above definitions, it is worth noticing that we do not allow the number $l>0$ to be sufficiently large: in some concrete situations, it is crucial to allow the number $l>0$ to be sufficiently small; we will explain this fact by two illustrative examples. First, let us consider Definition \ref{w-orlicz}(i). Suppose that $p(x)\equiv p\in [1,\infty)$ and there exists an absolute constant $c>0$ such that for each $l>0$ and $\tau \in I$ we have
$$
\sup_{t\in I}\phi \bigl( \bigl\| f(\cdot+\tau) -f(\cdot)\bigr\|\bigr)_{L^{p(x)}[t,t+l]}\leq c.
$$
Then it simply follows that the function $f(\cdot)$ is equi-Weyl-$(p,\phi,\psi)$-almost periodic provided that $\lim_{l\rightarrow 0+}\psi(l)=+\infty.$ Second, suppose that $f\in L^{\infty}(I:X).$ Then $f(\cdot)$ is equi-Weyl-$(p,x,1)$-almost periodic for any
$p\in {\mathcal D}(I),$ which can be simply approved by considering the case of constant coefficient $p(x) \equiv p^{+}$ and the choice $l=l(\epsilon)=\epsilon.$
\end{itemize}
\end{rem}

In order to ensure the translation invariance of Weyl spaces with variable exponent, we need to follow a slightly different approach (\cite{AP}-\cite{toka-marek-prim}):

\begin{defn}\label{w-orliczti}
Suppose that condition \mbox{(B)} holds, $f : I \rightarrow X$ and 
$\phi ( \| f(\cdot  l+t+\tau) -f(t+\cdot l)\|)\in L^{p(x)}([0,1])$
for any $\tau \in I,$ $t\in I$ and $l>0.$
\begin{itemize}
\item[(i)] It is said that the function $f(\cdot)$ is equi-Weyl-$[p,\phi,F]$-almost periodic, $f\in e-W_{ap}^{[p,\phi,F]}(I:X)$ for short, iff for each $\epsilon>0$ we can find two real numbers $l>0$ and $L>0$ such that any interval $I'\subseteq I$ of length $L$ contains a point $\tau \in  I'$ such that
\begin{align*}
e-\|f\|_{[p,\phi,F,\tau]}:=\sup_{t\in I}\Biggl[ F(l,t)\Bigl[\phi \bigl( \bigl\| f(\cdot  l+t+\tau) -f(t+\cdot l)\bigr\|\bigr)_{L^{p(\cdot)}[0,1]}\Bigr]\Biggr] \leq \epsilon .
\end{align*}
\item[(ii)] It is said that the function $f(\cdot)$ is Weyl-$[p,\phi,F]$-almost periodic, $f\in W_{ap}^{[p,\phi,F]}(I: X)$ for short, iff for each $\epsilon>0$ we can find a real number $L>0$ such that any interval $I'\subseteq I$ of length $L$ contains a point $\tau \in  I'$ such that
\begin{align*}
\|f\|_{[p,\phi,F,\tau]}:=\limsup_{l\rightarrow \infty} \sup_{t\in I}\Biggl[ F(l,t) \Bigl[\phi \bigl(\bigl \| f(\cdot l+t+\tau) -f(t+\cdot l)\bigr\|\bigr)_{L^{p(\cdot)}[0,1]} \Bigr]\Biggr]  \leq \epsilon .
\end{align*}
\end{itemize}
\end{defn}

\begin{defn}\label{w-orlicztist}
Suppose that condition \mbox{(B)} holds, $f : I \rightarrow X$ and 
$\| f(\cdot  l+t+\tau) -f(t+\cdot l)\|\in L^{p(x)}([0,1])$
for any $\tau \in I,$ $t\in I$ and $l>0.$
\begin{itemize}
\item[(i)] It is said that the function $f(\cdot)$ is equi-Weyl-$[p,\phi,F]_{1}$-almost periodic, $f\in e-W_{ap}^{[p,\phi,F]_{1}}(I:X)$ for short, iff for each $\epsilon>0$ we can find two real numbers $l>0$ and $L>0$ such that any interval $I'\subseteq I$ of length $L$ contains a point $\tau \in  I'$ such that
\begin{align*}
e-\|f\|_{[p,\phi,F,\tau]_{1}}:=\sup_{t\in I}\Biggl[ F(l,t) \phi\Bigl[ \bigl( \bigl\| f(\cdot  l+t+\tau) -f(t+\cdot l)\bigr\|\bigr)_{L^{p(\cdot)}[0,1]}\Bigr]\Biggr] \leq \epsilon .
\end{align*}
\item[(ii)] It is said that the function $f(\cdot)$ is Weyl-$[p,\phi,F]_{2}$-almost periodic, $f\in W_{ap}^{[p,\phi,F]_{2}}(I: X)$ for short, iff for each $\epsilon>0$ we can find a real number $L>0$ such that any interval $I'\subseteq I$ of length $L$ contains a point $\tau \in  I'$ such that
\begin{align*}
\|f\|_{[p,\phi,F,\tau]_{1}}:=\limsup_{l\rightarrow \infty} \sup_{t\in I}\Biggl[ F(l,t) \phi \Bigl[ \bigl(\bigl \| f(\cdot l+t+\tau) -f(t+\cdot l)\bigr\|\bigr)_{L^{p(\cdot)}[0,1]} \Bigr]\Biggr]  \leq \epsilon .
\end{align*}
\end{itemize}
\end{defn}

\begin{defn}\label{w-orlicztistr}
Suppose that condition \mbox{(B)} holds, $f : I \rightarrow X$ and 
$ \| f(\cdot  l+t+\tau) -f(t+\cdot l)\|\in L^{p(x)}([0,1])$
for any $\tau \in I,$ $t\in I$ and $l>0.$
\begin{itemize}
\item[(i)] It is said that the function $f(\cdot)$ is equi-Weyl-$[p,\phi,F]_{2}$-almost periodic, $f\in e-W_{ap}^{[p,\phi,F]_{2}}(I:X)$ for short, iff for each $\epsilon>0$ we can find two real numbers $l>0$ and $L>0$ such that any interval $I'\subseteq I$ of length $L$ contains a point $\tau \in  I'$ such that
\begin{align*}
e-\|f\|_{[p,\phi,F,\tau]_{2}}:=\sup_{t\in I}\phi \Biggl[ F(l,t)\Bigl[ \bigl( \bigl\| f(\cdot  l+t+\tau) -f(t+\cdot l)\bigr\|\bigr)_{L^{p(\cdot)}[0,1]}\Bigr]\Biggr] \leq \epsilon .
\end{align*}
\item[(ii)] It is said that the function $f(\cdot)$ is Weyl-$[p,\phi,F]_{2}$-almost periodic, $f\in W_{ap}^{[p,\phi,F]_{2}}(I: X)$ for short, iff for each $\epsilon>0$ we can find a real number $L>0$ such that any interval $I'\subseteq I$ of length $L$ contains a point $\tau \in  I'$ such that
\begin{align*}
\|f\|_{[p,\phi,F,\tau]_{2}}:=\limsup_{l\rightarrow \infty} \sup_{t\in I}\phi \Biggl[ F(l,t) \Bigl[ \bigl(\bigl \| f(\cdot l+t+\tau) -f(t+\cdot l)\bigr\|\bigr)_{L^{p(\cdot)}[0,1]} \Bigr]\Biggr]  \leq \epsilon .
\end{align*}
\end{itemize}
\end{defn}

\begin{rem}\label{zajev}
\begin{itemize}
\item[(i)]
Let $p\in {\mathcal P}([0,1]),$ let $I={\mathbb R}$ or $I=[0,\infty),$ and let a function $f\in L_{S}^{p(x)}(I:X)$ be Stepanov $p(x)$-almost periodic. Then it readily follows that $f(\cdot)$ 
is equi-Weyl-$[p,\phi,F]$-almost periodic with $\phi(x)\equiv x$ and $F(l,t)\equiv 1.$ 
\item[(ii)] In the case that $p(x)\equiv p\in [1,\infty),$ it can be simply verified that the class of (equi)-Weyl-$[p,\phi,[l/\psi(l)]^{1/p}]$-almost periodic functions, resp. 
(equi)-Weyl-$[p,\phi,[l/\psi(l)]^{1/p}]_{2}$-almost periodic functions, coincides with the class of (equi)-Weyl-$(p,\phi,\psi)$-almost periodic functions, resp. (equi)-Weyl-$(p,\phi,\psi)_{2}$-almost periodic functions. The classes of (equi)-Weyl-$[p,\phi,[l/\psi(l)]^{1/p}]_{1}$-almost periodic functions and (equi)-Weyl-$(p,\phi,\psi)_{1}$-almost periodic functions coincide
provided that $\phi(cx)=c\phi(x)$ for all $c,\ x\geq 0.$
\item[(iii)] It can be simply verified that the validity of condition
\begin{itemize}
\item[(D):] For any $\tau_{0}\in I$ there exists $c>0$ such that
$$
\frac{F(l,t)}{F(l,t+\tau_{0})}\leq c,\quad t\in I,\ l>0
$$
\end{itemize}
implies that the spaces $(e-)W_{ap}^{[p,\phi,F]}(I:X)$ and $(e-)W_{ap}^{[p,\phi,F]_{1}}(I:X)$
are translation invariant; this particularly holds provided the function $F(l,t)$ does not depend on the variable $t.$ Furthermore, the space $(e-)W_{ap}^{[p,\phi,F]_{2}}(I:X)$ is translation invariant provided condition
\begin{itemize}
\item[(D)':] For any $\tau_{0}\in I$ there exists $c>0$ such that 
$$
\phi(F(l,t)x)\leq c\phi \bigl(F(l,t+\tau_{0})x\bigr),\quad x\geq 0,\ t\in I,\ l>0.
$$
\end{itemize}
\item[(iv)] If $p,\ q\in {\mathcal P}([0,1])$ and $q(x)\leq p(x)$ for a.e. $x\in [0,1],$ then Lemma \ref{aux}(ii) yields that any (equi)-Weyl-$[p,\phi,F]$-almost periodic function is 
(equi)-Weyl-$[q,\phi,F]$-almost periodic. Furthermore, condition $x,\ y\geq 0$ and $x\leq cy$ implies $\phi(x)\leq c\phi(y),$ resp.  $x,\ y\geq 0$ and $x\leq cy$ implies $\phi(F(l,t)x)\leq c\phi(F(l,t)y)$ for all $l>0$ and $t\in I,$ ensures that any (equi)-Weyl-$[p,\phi,F]_{1}$-almost periodic function is 
(equi)-Weyl-$[q,\phi,F]_{1}$-almost periodic, resp. any (equi)-Weyl-$[p,\phi,F]_{2}$-almost periodic function is 
(equi)-Weyl-$[q,\phi,F]_{2}$-almost periodic.
\item[(v)] It is clear that, if $f(\cdot)$ is an (equi)-Weyl-$[p,\phi,F]$-almost periodic function, resp. (equi)-Weyl-$[p,\phi,F]_{1}$-almost periodic function, and 
$F(l,t)\geq F_{1}(l,t)$ for every $l>0$ and $t\in I,$ then $f(\cdot)$ is (equi)-Weyl-$[p,\phi,F_{1}]$-almost periodic, resp. (equi)-Weyl-$[p,\phi,F_{1}]_{1}$-almost periodic. Furthermore, any (equi)-Weyl-$[p,\phi,F]_{2}$-almost periodic function is (equi)-Weyl-$[p,\phi,F]_{2}$-almost periodic provided that $F(l,t)\geq F_{1}(l,t)$ for every $l>0,$ $t\in I$ and $\phi(\cdot)$ is monotonically increasing, or $F(l,t)\leq F_{1}(l,t)$ for every $l>0,$ $t\in I$ and $\phi(\cdot)$ is monotonically decreasing.
\item[(vi)] If $f(\cdot)$ is an (equi)-Weyl-$[p,\phi,F]$-almost periodic function,  $\phi_{1} ( \| f(\cdot  l+t+\tau) -f(t+\cdot l)\|)$
is measurable for any $\tau \in I,$ $t\in I,$ $l>0,$ and $0\leq \phi_{1}\leq \phi,$
then Lemma \ref{aux}(iii) yields that $f(\cdot)$ is an (equi)-Weyl-$[p,\phi_{1},F]$-almost periodic. Furthermore, if $0\leq \phi_{1}\leq \phi ,$ only, and 
$f(\cdot)$ is an (equi)-Weyl-$[p,\phi,F]_{i}$-almost periodic function, then $f(\cdot)$ is an (equi)-Weyl-$[p,\phi_{1},F]_{i}$-almost periodic function, where $i=1,2.$
\end{itemize}
\end{rem}

In the case that the function $\phi(\cdot)$ is convex and $p(x)\equiv 1$, we have the following proposition which can be shown following the lines of the proof
of Proposition \ref{stevate}:

\begin{prop}\label{nasalo}
Suppose that $\phi(\cdot)$ is convex, $p(x)\equiv 1,$ $f : I \rightarrow X$ and 
$\| f(\cdot  l+t+\tau) -f(t+\cdot l)\|\in L^{p(x)}([0,1])$
for any $\tau \in I,$ $t\in I$ and $l>0.$ Then the following holds:
\begin{itemize}
\item[(i)] $f\in (e-)W_{ap}^{[1,\phi,F]} \Rightarrow f\in (e-)W_{ap}^{[1,\phi,F]_{1}}.$
\item[(ii)] If condition $\emph{(C)}$ holds, then $f\in (e-)W_{ap}^{[1,\phi,\varphi \circ F]} \Rightarrow f\in (e-)W_{ap}^{[1,\phi,F]_{2}}.$ 
\end{itemize} 
\end{prop}

Regarding Proposition \ref{stevate} and Proposition \ref{nasalo}, it should be observed that the reverse inclusions and inequalities can be obtained assuming condition
\begin{itemize}
\item[(C)':] $\phi(\cdot)$ is concave and there exists a function $\varphi :[0,\infty) \rightarrow [0,\infty)$ such that $\phi(lx)\geq \varphi(l) \phi(x)$ for all $l>0$ and $x\geq 0.$  
\end{itemize}

It is clear that any (equi-)Weyl-$p$-almost periodic function $f(\cdot)$ 
is (equi-)Weyl-$(p,\phi,\psi)$-almost periodic with $p(x)\equiv p \in [1,\infty),$ $\phi(x)=x$, $\psi(l)=l.$ Concerning this observation, we wish to present 
two illustrative examples:

\begin{example}\label{wert}
Let us recall (see, e.g., Example 4.27 in the survey article \cite{deda} by J. Andres, A. M. Bersani, R. F. Grande and the monograph \cite{nova-mono}) that the function $g(\cdot):=\chi_{[0,1/2]}(\cdot)$ is equi-Weyl-$p$-almost periodic for any $p\in [1,\infty)$ but not Stepanov almost periodic.
Since for each $l,\ \tau \in {\mathbb R}$ we have 
$$
\Biggl(\sup_{t\in {\mathbb R}}\int_{t}^{t+l}|f(x+\tau)-f(x)|^{p}\, dx\Biggr)^{1/p}\leq 1,
$$ 
it can be easily seen that the function $g(\cdot)$ is equi-Weyl-$(p,x,\psi)$-almost periodic for any function $\psi : (0,\infty) \rightarrow (0,\infty)$ such that $\lim_{l\rightarrow +\infty}\psi(l)=+\infty ;$ moreover, for each $\epsilon \in (0,1/2)$ we can always find
$t\in {\mathbb R}$ such that 
$$
\int_{t}^{t+1}|f(x+\tau)-f(x)|^{p}\, dx >\epsilon,\quad \tau>\epsilon.
$$
Hence, the function $g(\cdot)$ cannot be equi-Weyl-$(p,x,l^{0})$-almost periodic. Taking into account Remark \ref{zajev1}(iii) and the above conclusions, we get that
$g(\cdot)$ is equi-Weyl-$(p,x,l^{\sigma})$-almost periodic iff $\sigma \neq 0.$
\end{example}

\begin{example}\label{wert1}
Let us recall (\cite[Example 4.29]{deda}, \cite{nova-mono}) that the Heaviside function $g(\cdot):=\chi_{[0,\infty)}(\cdot)$ is not equi-Weyl-$1$-almost periodic but it is Weyl-$p$-almost periodic for any number $p\in [1,\infty).$ Furthermore, it is not difficult to see that for each real number $\tau \in {\mathbb R}$ we have that $\sup_{t\in {\mathbb R}}(\int^{t+l}_{t}|f(x+\tau)-f(x)|^{p}\, dx)^{1/p}=|\tau|^{1/p}$ for any real number $l>|\tau|.$ This simply implies that the function $g(\cdot)$ is Weyl-$(p,x,\psi)$-almost periodic for any function $\psi : (0,\infty) \rightarrow (0,\infty)$
such that $\lim_{l\rightarrow +\infty}\psi(l)=+\infty $ as well as that $g(\cdot)$ cannot be Weyl-$(p,x,\psi)$-almost periodic for any function $\psi : (0,\infty) \rightarrow (0,\infty)$
such that $\limsup_{l\rightarrow +\infty}[\psi(l)]^{-1}>0$; in particular, $g(\cdot)$ is Weyl-$(p,x,l^{\sigma})$-almost periodic iff $\sigma>0.$ On the other hand, the function $g(\cdot)$ cannot be equi-Weyl-$(p,x,\psi)$-almost periodic for any function $\psi : (0,\infty) \rightarrow (0,\infty);$ in actual fact, if we suppose contrary, then the equation \eqref{profice} is violated with $|\tau|^{1/p}>\epsilon \psi(l)^{1/p}.$ See also \cite[Example 2.11.15-Example 2.11.17]{nova-mono}.
\end{example}

\section{Weyl ergodic components with variable exponents}\label{weylerg}

Unless stated otherwise, in this section we assume that  $p\in {\mathcal P}([0,\infty)),$
$\phi : [0,\infty) \rightarrow [0,\infty)$ and $F: (0,\infty) \times [0,\infty) \rightarrow (0,\infty).$ % By $L_{loc}^{p(x)}([0,\infty ) : X)$ we denote the space of all locally $p(x)$-integrable $X$-valued functions defined on $[0,\infty).$
 In the following three definitions, we extend the notion of an (equi-)Weyl-$p$-vanishing function introduced in \cite{weyl}, where the case $p(x)\equiv p\in [1,\infty),$ $F(l,t)\equiv l^{(-1)/p}$ and $\phi(x)\equiv x$ has been considered:

\begin{defn}\label{stea-weyl0px}
\begin{itemize}
\item[(i)] It is said that a function $q : [0,\infty) \rightarrow X$ is equi-Weyl-$(p,\phi,F)$-vanishing iff $\phi(\|q(t+\cdot)\| )\in L^{p(\cdot)}[x,x+l]$ for all $t,\ x,\ l>0$ and 
\begin{align}\label{radnasebipx01}
\lim_{l\rightarrow +\infty}\, \limsup_{t\rightarrow +\infty}\ \ \sup_{x\geq 0}\Bigl[ F(l,t)\bigl \| \phi (\|q(t+v)\| )\bigr\|_{L^{p(v)}[x,x+l]}\Bigr]=0.
\end{align}
\item[(ii)] It is said that a function $q : [0,\infty) \rightarrow X$ is Weyl-$(p,\phi,F)$-vanishing iff $\phi(q(t+\cdot))\in L^{p(\cdot)}[x,x+l]$ for all $t,\ x,\ l>0$ and
\begin{align}\label{radnasebipx012345}
\lim_{t\rightarrow +\infty}\, \limsup_{l\rightarrow +\infty}\ \ \sup_{x\geq 0}\Bigl[ F(l,t)\bigl \| \phi(\|q(t+v)\| )\bigr\|_{L^{p(v)}[x,x+l]}\Bigr]=0.
\end{align}
\end{itemize}
\end{defn}

\begin{defn}\label{stea-weyl0px11}
\begin{itemize}
\item[(i)] It is said that a function $q : [0,\infty) \rightarrow X$ is equi-Weyl-$(p,\phi,F)_{1}$-vanishing iff $q(t+\cdot)\in L^{p(\cdot)}[x,x+l]$ for all $t,\ x,\ l>0$ and
\begin{align}\label{radnasebipx1}
\lim_{l\rightarrow +\infty}\, \limsup_{t\rightarrow +\infty}\ \ \sup_{x\geq 0}\Bigl[ F(l,t)\phi\Bigl(\bigl \| q(t+v)\bigr\|_{L^{p(v)}[x,x+l]}\Bigr)\Bigr]=0.
\end{align}
\item[(ii)] It is said that a function $q : [0,\infty) \rightarrow X$ is Weyl-$(p,\phi,F)_{1}$-vanishing iff $q(t+\cdot)\in L^{p(\cdot)}[x,x+l]$ for all $t,\ x,\ l>0$ and
\begin{align}\label{radnasebipx21}
\lim_{t\rightarrow +\infty}\, \limsup_{l\rightarrow +\infty}\ \ \sup_{x\geq 0}\Bigl[ F(l,t)\phi\Bigl(\bigl \| q(t+v)\bigr\|_{L^{p(v)}[x,x+l]}\Bigr)\Bigr]=0.
\end{align}
\end{itemize}
\end{defn}

\begin{defn}\label{stea-weyl0px2}
\begin{itemize}
\item[(i)] It is said that a function $q : [0,\infty) \rightarrow X$ is equi-Weyl-$(p,\phi,F)_{2}$-vanishing iff $q(t+\cdot)\in L^{p(\cdot)}[x,x+l]$ for all $t,\ x,\ l>0$  and
\begin{align}\label{radnasebipx}
\lim_{l\rightarrow +\infty}\, \limsup_{t\rightarrow +\infty}\ \ \sup_{x\geq 0}\phi\Bigl[ F(l,t)\bigl \| q(t+v)\bigr\|_{L^{p(v)}[x,x+l]}\Bigr]=0.
\end{align}
\item[(ii)] It is said that a function $q : [0,\infty) \rightarrow X$ is Weyl-$(p,\phi,F)_{2}$-vanishing iff $q(t+\cdot)\in L^{p(\cdot)}[x,x+l]$ for all $t,\ x,\ l>0$ and
\begin{align}\label{radnasebipxx}
\lim_{t\rightarrow +\infty}\, \limsup_{l\rightarrow +\infty}\ \ \sup_{x\geq 0}\phi\Bigl[ F(l,t)\bigl \| q(t+v)\bigr\|_{L^{p(v)}[x,x+l]}\Bigr]=0.
\end{align}
\end{itemize}
\end{defn}

Denote by $W^{p(x)}_{\phi, F,0}([0,\infty):X)$ and $e-W^{p(x)}_{\phi, F, 0}([0,\infty):X)$ [$W^{p(x);1}_{\phi, F,0}([0,\infty):X)$ and $e-W^{p(x);1}_{\phi, F, 0}([0,\infty):X)$/$W^{p(x);2}_{\phi, F,0}([0,\infty):X)$ and $e-W^{p(x);2}_{\phi, F, 0}([0,\infty):X)$] the sets consisting of all Weyl-$(p,\phi,F)$-vanishing functions and equi-Weyl-$(p,\phi, F)$-vanishing functions [Weyl-$(p,\phi,F)_{1}$-vanishing functions and equi-Weyl-$(p,\phi, F)_{1}$-vanishing functions/Weyl-$(p,\phi,F)_{2}$-vanishing functions and equi-Weyl-$(p,\phi, F)_{2}$-vanishing functions], respectively. In the case that $p(x)\equiv p\in [1,\infty),$ $F(l,t)\equiv l^{(-1)/p}$ and $\phi(x)\equiv x,$ the above classes coincide and we denote them by $W^{p}_{0}([0,\infty):X) $ and $e-W^{p}_{0}([0,\infty):X).$ These classes 
are very general and we want only to recall that, for instance,
an equi-Weyl-$p$-vanishing function $q(\cdot)$ need not be bounded as $t\rightarrow +\infty$ (\cite{weyl}).

A great number of very simple examples can be constructed in order to show that, in general case, the limit $\lim_{t\rightarrow +\infty}\sup_{x\geq 0}[ F(l,t) \| \phi(\|q(t+v)\|)\|_{L^{p(v)}[x,x+l]}]$ in the equation \eqref{radnasebipx01} does not exist for any fixed number $l>0$; the same holds for the equations \eqref{radnasebipx012345}-\eqref{radnasebipxx}. The question when these limits exist is meaningful but it will not be analyzed here for the sake of brevity.

Further on, we have the following observation:

\begin{rem}\label{jaz}
\begin{itemize}
\item[(i)]
Suppose that the function $\phi(\cdot)$ is monotonically increasing and satisfies that for each scalars $\alpha,\ \beta\geq 0$ there exists a finite real number 
$\pi(\alpha,\beta)>0$
such that, for every non-negative real numbers
$x,\ y\geq 0,$ we have 
$$
\phi(\alpha x+\beta y)\leq \pi(\alpha,\beta)[\phi(x)+\phi(y)].
$$ 
Then (equi-)Weyl-$(p,\phi,F)$-vanishing functions and (equi-)Weyl-$(p,\phi,F)_{i}$-vanishing functions, where $i=1,2,$ form a vector space. 
\item[(ii)] If the function $F(l,t)$ satisfies condition (D), resp. (D)', then the space of (equi-)Weyl-$(p,\phi,F)$-vanishing functions and the space of (equi-)Weyl-$(p,\phi,F)_{1}$-vanishing functions, resp. 
the space of (equi-)Weyl-$(p,\phi,F)_{2}$-vanishing functions, are translation invariant. 
\end{itemize}
\end{rem}

In this paper, we will not follow the approach obeyed in \cite{AP} and previous section, with the basic assumption
$p\in {\mathcal P}([0,1]).$ 
With regards to this question, 
we will present only one illustrative example:

\begin{example}\label{lund-quarter}
Suppose that $p\in {\mathcal P}([0,1]).$
Let us recall that the space of Stepanov $p(\cdot)$-vanishing functions (see  \cite{AP}), denoted by $S^{p(x)}_{0}([0,\infty):X),$ is consisted 
of those functions
$q\in  L_{S}^{p(x)}([0,\infty): X)$ such that $\hat{q}\in C_{0}([0,\infty) : L^{p(x)}([0,1]:X)).$ The notion of space $S^{p(x)}_{0}([0,\infty):X)$ can be extended in many other ways; for example:
\begin{itemize}
\item[(i)] Let $\phi : (0,\infty) \rightarrow (0,\infty)$ and $G : (0,\infty) \rightarrow (0,\infty).$ Then we say that a function $q(\cdot)$ belongs to the space $S^{p(\cdot)}_{\phi, G,0}([0,\infty):X)$ iff $\phi(\|q(t+\cdot)\|) \in L^{p(\cdot)}[0,1]$ for all $t\geq 0$ 
and
$$
\lim_{t\rightarrow +\infty}G(t)\bigl \| \phi(\|q(t+v)\|) \bigr\|_{L^{p(v)}[0,1]}=0.
$$ 
In this part, as well as in parts (ii) and (iii), we will use the $1$-periodic extension of function $p(\cdot)$ to the non-negative real axis, denoted henceforth by $p_{1}(\cdot).$ Then the class $S^{p(\cdot)}_{\phi, G,0}([0,\infty):X)$ is contained in the class of
equi-Weyl-$(p_{1},\phi,F)$-vanishing functions with a suitable chosen function $F(l,t).$ More precisely, let a number $\epsilon>0$ be fixed. Then there exists a sufficiently large real number $t_{0}>0$ such that 
$\| \phi(q(t+v)) \|_{L^{p(v)}[0,1]}<\epsilon G(t)^{-1}$ for all numbers $t\geq t_{0}.$ This implies that, for every
$t\geq t_{0},$ $x\geq 0$ and $m\in {\mathbb N}_{0},$ we have
$$
\int^{1}_{0}\varphi_{p(v)}\Bigl(\phi(\|q(t+v+\lfloor x \rfloor +m)\|) /\bigl[\epsilon G(t)^{-1}\bigr] \Bigr)\, dv \leq 1.
$$ 
Using the inequality ($x\geq 0,$ $l>0$)
\begin{align*}
\int^{x+l}_{x}&\varphi_{p_{1}(v)}\Bigl( \phi(\|q(t+v)\|)/\bigl[\epsilon G(t)^{-1}\bigr] \Bigr)\, dv 
\\ & \leq \sum_{k=0}^{l}\int^{\lfloor x \rfloor+k+1}_{\lfloor x \rfloor+k}\varphi_{p_{1}(v)}\Bigl(\phi(\| q(t+v)\|)/\bigl[\epsilon G(t)^{-1}\bigr] \Bigr)\, dv,
\end{align*}
the above yields
\begin{align*}
\int^{x+l}_{x}& \varphi_{p_{1}(v)}\Bigl( \phi(\| q(t+v)\|)/\bigl[\epsilon G(t)^{-1}\bigr] \Bigr)\, dv \leq l+1,\, \mbox{ i.e., }
\\ & \int^{x+l}_{x}\frac{1}{l+1}\varphi_{p_{1}(v)}\Bigl(\phi(\| q(t+v)\|)/\bigl[\epsilon G(t)^{-1}\bigr] \Bigr)\, dv \leq 1.
\end{align*}
Since
$$
\varphi_{p_{1}(v)}\Bigl( \phi(\| q(t+v)\|)/\bigl[\epsilon (l+1)G(t)^{-1}\bigr] \Bigr) \leq \frac{1}{l+1}\varphi_{p_{1}(v)}\Bigl( \phi(\| q(t+v)\| )/\bigl[\epsilon G(t)^{-1}\bigr] \Bigr),
$$
the above implies $\| \phi(\| q(t+v)\|) \|_{L^{p(v)}[x,x+l]}<\epsilon G(t)^{-1}(1+l)$ for all $t\geq t_{0},$ $x\geq 0$ and $l>0.$ Hence, the required conclusion holds provided that there exists a finite real constant $C>0$ such that
$$
| F(l,t)G(t)^{-1}(1+l)|\leq C,\quad l>0,\ t>0.
$$
\item[(ii)] Let $\phi : (0,\infty) \rightarrow (0,\infty)$ and $G : (0,\infty) \rightarrow (0,\infty).$ Then we say that a function $q(\cdot)$ belongs to the space $S^{p(\cdot)}_{\phi, G,0;1}([0,\infty):X)$ iff $q(t+\cdot) \in L^{p(\cdot)}[0,1]$ for all $t\geq 0$ 
and
$$
\lim_{t\rightarrow +\infty}G(t)\phi\Bigl(\bigl \| q(t+v) \bigr\|_{L^{p(v)}[0,1]}\Bigr)=0.
$$ 
Then the class $S^{p(\cdot)}_{\phi, G,0;1}([0,\infty):X)$ is contained in the class of
equi-Weyl-$(p_{1},\phi,F)_{1}$-vanishing functions with a suitable chosen function $F(l,t).$ Arguing as in (i), this holds provided that, for example, $\sup \phi^{-1}([0,G(t)^{-1}])<\infty$ and
$$
\lim_{l\rightarrow +\infty}\, \limsup_{t\rightarrow +\infty}\ \ F(l,t)(l+1)\sup \phi^{-1}\Bigl(\bigl[0,G(t)^{-1}\bigr]\Bigr)=0.
$$
\item[(iii)] Let $\phi : (0,\infty) \rightarrow (0,\infty)$ and $G : (0,\infty) \rightarrow (0,\infty).$ Then we say that a function $q(\cdot)$ belongs to the space $S^{p(\cdot)}_{\phi, G,0;2}([0,\infty):X)$ iff $q(t+\cdot) \in L^{p(\cdot)}[0,1]$ for all $t\geq 0$ 
and
$$
\lim_{t\rightarrow +\infty}\phi\Bigl(G(t)\bigl \| \phi(q(t+v)\bigr\|_{L^{p(v)}[0,1]}\Bigr)=0.
$$ 
Then the class $S^{p(\cdot)}_{\phi, G,0;2}([0,\infty):X)$ is contained in the class of
equi-Weyl-$(p_{1},\phi,F)_{2}$-vanishing functions with a suitable chosen function $F(l,t).$ Arguing as in (i), this holds provided that, for example, the function $\phi(\cdot)$ is monotonically increasing, $\sup \phi^{-1}([0,1]) <+\infty$ and 
$$
\lim_{l\rightarrow +\infty}\, \limsup_{t\rightarrow +\infty}\ \ \phi \Bigl(F(l,t)G(t)^{-1}(1+l)\sup \phi^{-1}([0,1])\Bigr)=0.
$$
\end{itemize}
\end{example}

An analogue of Proposition \ref{stevate} can be proved for  (equi-)Weyl-$(p,\phi,F)$-vanishing functions and (equi-)Weyl-$(p,\phi,F)_{i}$-vanishing functions, provided that the function $\phi(\cdot)$ is convex and $q(v)\equiv 1.$ Furthermore, an analogue of Remark \ref{zajev1}(i)-(ii) can be formulated for  (equi-)Weyl-$(p,\phi,F)$-vanishing functions and (equi-)Weyl-$(p,\phi,F)_{i}$-vanishing functions. Concerning Lemma \ref{aux}(ii) and Remark \ref{zajev}(v), it should be noted that the embedding type result established in already mentioned \cite[Corollary 3.3.4]{variable}
for scalar-valued functions (see also Lemma \ref{aux}(ii)) enables one to see that the following expected result holds true:

\begin{prop}\label{tirsen}
Suppose $r,\ p\in {\mathcal P}([0,\infty))$ and $1\leq r(x) \leq p(x)$ for a.e. $x\geq 0.$ 
Let $F_{1}(l,t)=2\max(l^{\mbox{essinf} (1/r(x)-1/p(x))},l^{\mbox{esssup} (1/r(x)-1/p(x))})F(l,t)$ or $F_{1}(l,t)=2(1+l)F(l,t)$ for all $l>0$ and $t\geq 0.$
Then we have:
\begin{itemize}
\item[(i)] If the function $q(\cdot)$ is (equi-)Weyl-$(r,\phi,F)$-vanishing provided that $q(\cdot)$ is (equi-)Weyl-$(p,\phi,F_{1})$-vanishing. 
\item[(ii)] Suppose that there exists a function $\varphi : [0,\infty) \rightarrow [0,\infty)$ such that $\phi(cx)\leq \varphi(c)\phi(x)$ for all $c\geq 0$ and $x\geq 0.$ 
Let $F_{2}(l,t)=\varphi(2(1+l))F(l,t)$ or $F_{1}(l,t)=\varphi(2\max(l^{\mbox{essinf} (1/r(x)-1/p(x))},l^{\mbox{esssup} (1/r(x)-1/p(x))}))F(l,t)$ for $l>0$ and $t\geq 0.$
Then the function $q(\cdot)$ is (equi-)Weyl-$(r,\phi,F)_{1}$-vanishing provided that $q(\cdot)$ is (equi-)Weyl-$(p,\phi,F_{2})_{1}$-vanishing. 
\item[(iii)]  If $\phi(\cdot)$ is monotonically increasing, then the function $q(\cdot)$ is (equi-)Weyl-$(r,\phi,F)_{2}$-vanishing provided that $q(\cdot)$ is (equi-)Weyl-$(p,\phi,F_{1})_{2}$-vanishing. 
\end{itemize}
\end{prop}

The case of constant coefficients $1\leq r\leq p$ also deserves attention, when the choices $F_{1}(l,t)=l^{1/r-1/p}F(l,t)$ in (i), (iii) and $F_{1}(l,t)=\varphi(l^{1/r-1/p})F(l,t)$ in (ii)  can be made.

We continue by reexaming the conclusions established in \cite[Example 4.5, Example 4.6]{weyl}:

\begin{example}\label{selja-sam}
Define
$$
q(t):=\sum_{n=0}^{\infty}\chi_{[n^{2},n^{2}+1]}(t),\quad t\geq 0.
$$
Then we know that $\hat{q}\notin C_{0}([0,\infty) : L^{p}([0,1]: {\mathbb C}))$ and the function $q(\cdot)$
is equi-Weyl-$p$-almost periodic for any exponent $p\geq 1;$ see
\cite[Example 4.5]{weyl}.  In this example, we have proved the estimate
$$
\Biggl(\int_{x}^{x+l}\bigl \| q(t+v)\bigr\|^{p}\, dv\Biggr)^{1/p} \leq \Biggl(2+\frac{l}{\sqrt{t}+\sqrt{l}}\Biggr)^{1/p}\leq 2+ \Biggl(\frac{l}{\sqrt{t}+\sqrt{l}}\Biggr)^{1/p},
$$
for any $x\geq 0,\ t\geq 0,\ l>0,$
so that 
the function $q(\cdot)$ is equi-Weyl-$(p,x,F)$-vanishing provided that 
$$
\lim_{l\rightarrow +\infty}\, \limsup_{t\rightarrow +\infty}\ \ F(l,t)\Biggl[2+\Biggl(\frac{l}{\sqrt{t}+\sqrt{l}}\Biggr)^{1/p}\Biggr]=0.
$$
In particular, this holds for function $F(l,t)=l^{\sigma},$ where $\sigma<0.$
\end{example}

\begin{example}\label{selja-samojed}
Define
$$
q(t):=\sum_{n=0}^{\infty}\sqrt{n}\chi_{[n^{2},n^{2}+1]}(t),\quad t\geq 0.
$$
Then we know that
the function $q(\cdot)$ is not equi-Weyl-$p$-vanishing for any exponent $p\geq 1$ as well as that the function $q(\cdot)$
is Weyl-$p$-vanishing for any exponent $p\geq 1;$ see \cite[Example 4.6]{weyl}. In this example, we have proved the estimate
$$
\Biggl(\int_{x}^{x+l}\bigl \| q(t+v)\bigr\|^{p}\, dv\Biggr)^{1/p} \leq \bigl(l+t\bigr)^{1/2p},\quad x\geq 0,\ t\geq 0,\ l>0,
$$
so that 
the function $q(\cdot)$ is Weyl-$(p,x,F)$-vanishing provided that 
$$
\lim_{t\rightarrow +\infty}\, \limsup_{l\rightarrow +\infty}\ \ F(l,t)\bigl(l+t\bigr)^{1/2p}=0.
$$
In particular, this holds for function $F(l,t)=l^{\sigma},$ where $\sigma<(-1)/2p.$
\end{example}

We will present one more illustrative example:

\begin{example}\label{more}
Suppose that $(a_{n})_{n\in {\mathbb N}}$ and $(b_{n})_{n\in {\mathbb N}}$ are two sequences of positive real numbers such that $(a_{n})_{n\in {\mathbb N}}$ is strictly monotonically increasing, $\lim_{n\rightarrow +\infty}(a_{n+1}-a_{n})=+\infty$
and $\lim_{n\rightarrow +\infty}\phi(b_{n})=0.$ Let $q : [0,\infty) \rightarrow (0,\infty)$ be defined by $q(t):=b_{n}$ iff $t\in [a_{n-1},a_{n})$ for some $n\in {\mathbb N},$ where $a_{0}:=0.$ If $p\in D_{+}([0,\infty)),$ $l>0$ and $t>0,$ then we have
\begin{align*}
\sup_{x\geq 0}&\Bigl[ F(l,t) \bigl\| \phi(q(t+v))\bigr\|_{L^{p(v)}[x,x+l]}\Bigr]
\\  & \leq \sup_{x\geq 0}\Bigl[ 2(1+l)F(l,t) \bigl\| \phi(q(t+\cdot))\bigr\|_{L^{p^{+}}[x,x+l]}\Bigr]
\\ & =\sup_{x\geq 0}\Bigl[ 2(1+l)F(l,t) \bigl\| \phi(q(\cdot))\bigr\|_{L^{p^{+}}[t+x,t+x+l]}\Bigr].
\end{align*} 
Assume, additionally, that there exists a function $ G: (0,\infty) \rightarrow (0,\infty)$ such that $F(l,t)\leq G(l)$ for all $l>0$ and $t>0.$ Since we have assumed that $\lim_{n\rightarrow +\infty}(a_{n+1}-a_{n})=+\infty,$ for each number $l>0$ we have 
$$
\limsup_{t\rightarrow +\infty}\sup_{x\geq 0}\Bigl[ 2(1+l)F(l,t) \bigl\| \phi(q(\cdot))\bigr\|_{L^{p^{+}}[t+x,t+x+l]}\Bigr]=0,
$$
because $\lim_{n\rightarrow +\infty}\phi(b_{n})=0$ and
$$
\bigl\| \phi(q(\cdot))\bigr\|_{L^{p^{+}}[t+x,t+x+l]}\leq l\max \bigl( \phi(b_{n}),\ \phi(b_{n+1}) \bigr),
$$
where $n\in {\mathbb N}$ is such that $x+t \leq a_{n}$ and $x+t+l \leq a_{n+1}.$ 
Therefore, the function $q(\cdot)$ is equi-Weyl-$(p,\phi,F)$-vanishing.
\end{example}

In \cite{weyl}, we have introduced a great number of various types of asymptotically Weyl almost periodic function spaces with constant exponent $p\geq 1.$ In order to relax our exposition, we will introduce here only one
general definition of an asymptotically Weyl almost periodic function with variable exponent, which extends the notion introduced in Definition \ref{sasasa}(ii):

\begin{defn}\label{asym-weyl-variable}
Let $h : [0,\infty) \rightarrow X.$ Then we say that $h(\cdot)$ is asymptotically Weyl almost periodic with variable exponent iff there exist two functions 
$g : [0,\infty) \rightarrow X$ and $q : [0,\infty) \rightarrow X$
such that $h(t)=g(t)+q(t)$ for a.e. $t\geq 0,$ $g(\cdot)$ belongs to some of function spaces introduced in Definition \ref{w-orlicz}-Definition \ref{w-orliczaq2} or Definition \ref{w-orliczti}-Definition \ref{w-orlicztistr}
and $q(\cdot)$ belongs to some of function spaces introduced in Definition \ref{stea-weyl0px}-Definition \ref{stea-weyl0px2} (with possibly different functions $p,\ p_{1};\ \phi,\ \phi_{1};\ F,\ F_{1}$ and the meaning clear).
\end{defn}

In \cite{sead-abas}, S. Abbas has introduced the notion of a Weyl $p$-pseudo
ergodic component ($p\geq 1$). We can extend this notion following the approach obeyed in the previous part of paper and provide certain 
extensions of \cite[Proposition 4.11]{weyl} in this context. Details can be left to the interested reader.

\section{Weyl almost periodicity with variable exponent and convolution products}\label{cp-inv}

In the analyses of (equi-)Weyl-$(p,\phi,F)$-almost periodic functions and  (equi-)Weyl-$[p,\phi,F]$-almost periodic functions, we will use the following conditions:
\begin{itemize}
\item[(A1):] $I={\mathbb R}$ or $I=[0,\infty),$ $\psi : (0,\infty) \rightarrow (0,\infty),$ $\varphi : [0,\infty) \rightarrow [0,\infty),$  $\phi :[0,\infty) \rightarrow [0,\infty) $ is a convex monotonically increasing function satisfying $\phi (xy)\leq \varphi(x)\phi(y)$ for all $x, \ y\geq 0,$ $ p\in {\mathcal P}(I).$
\item[(B1):] The same as (A) with the assumption $ p\in {\mathcal P}(I)$ replaced by $ p\in {\mathcal P}([0,1])$ therein.
\end{itemize}

\begin{thm}\label{jensen}
Suppose that condition \emph{(A1)} holds with $I={\mathbb R}$, $\check{g} : {\mathbb R} \rightarrow X$ is (equi-)Weyl-$(p,\phi,F)$-almost periodic and measurable, $F_{1} : (0,\infty) \times I \rightarrow (0,\infty),$ $p,\ q\in {\mathcal P}({\mathbb R}),$ $1/p(x) +1/q(x)=1,$
$(R(t))_{t> 0}\subseteq L(X,Y)$ is a strongly continuous operator family and $(a_{k})$ is a sequence of positive real numbers such that $\sum_{k=0}^{\infty}a_{k}=1.$ 
If for every real numbers $x,\ \tau \in {\mathbb R}$ we have 
\begin{align}\label{picni}
\int^{\infty}_{-x}\|R(v+x)\| \|\check{g}(v)\| \, dv <\infty ,
\end{align}
and if,
for every $t\in {\mathbb R}$ and $l>0,$ we have
\begin{align}\label{radio}
H(l,x):=\sum_{k=0}^{\infty}a_{k}\varphi(la_{k}^{-1}) \bigl\| \varphi\bigl(\| R(v+x) \| \bigr)\bigr\|_{L^{q(v)}[-x+kl,-x+(k+1)l]} F(l,-x+lk)^{-1}<\infty,
\end{align}
\begin{align}\label{okje}
\int^{t+l}_{t}\varphi_{p(x)}\Bigl(2  l^{-1} H(l,x)F_{1}(l,t)^{-1} \Bigr)\, dx \leq 1,
\end{align}
resp. if \eqref{radio} holds and there exists $l_{0}>0$ such that for all $l\geq l_{0}$ and $t\in {\mathbb R}$ we have \eqref{okje},
then the function $G: {\mathbb R} \rightarrow Y,$ given by
\begin{align}\label{wer}
G(x):=\int^{x}_{-\infty}R(x-s)g(s)\, ds,\quad x\in {\mathbb R},
\end{align}
is well-defined and (equi-)Weyl-$(p,\phi,F_{1})$-almost periodic.
\end{thm}

\begin{proof}
We will prove the theorem only for the class of equi-Weyl-$(p,\phi,F)$-almost periodic functions. Since 
$G(x)=\int^{\infty}_{-x}R(v+x)\check{g}(v)\, dv, $ $ x\in {\mathbb R},$ the estimate in \eqref{picni} 
shows that the function $G(\cdot)$ is well-defined and that the integral in definition of $G(x)$ converges absolutely ($ x\in {\mathbb R}$).  Furthermore, the same 
estimate shows that for each real number $\tau$ we have $\int^{\infty}_{-x}\|R(v+x)\| \|\check{g}(v+\tau)\| \, dv=\int^{\infty}_{-(x-\tau)}\|R(v+(x-\tau))\| \|\check{g}(v)\| \, dv <\infty,$ so that the integral in definition of $G(x+\tau)-G(x)$ converges absolutely ($ x\in {\mathbb R}$).
Let $\epsilon>0$ be a fixed real number. Then we can find two real numbers $l>0$ and $L>0$ such that any interval $I'\subseteq I$ of length $L$ contains a point $\tau \in  I'$ such that
\eqref{profice} holds for the function $\check{g}(\cdot),$ with the number $\tau$ replaced by the number $-\tau$ therein.
Using our assumptions from condition (A1), the
Jensen integral inequality applied to the function $\phi(\cdot)$ (see also condition \eqref{picni}), the fact that the functions $\phi(\cdot)$ and $\varphi_{p(x)}(\cdot)$ are monotonically increasing, \eqref{infinitever} and Lemma \ref{aux}(i), we get that for each real number $x\in {\mathbb R}$ the following holds:
\begin{align*}
& \varphi_{p(x)}\Bigl(\phi(\|G(x +\tau)-G(x)\|) /\lambda \Bigr)
\\& \leq \varphi_{p(x)}\Biggl( \phi \Bigl(\int^{\infty}_{
-x}\| R(v+x) \| \|\check{g}(v+\tau)-\check{g}(v)\|\, dv\Bigr)/\lambda\Biggr)
\\& =\varphi_{p(x)}\Biggl( \phi \Bigl(\sum_{k=0}^{\infty} a_{k}\int^{-x+(k+1)l}_{
-x+kl}a_{k}^{-1}\| R(v+x) \| \|\check{g}(v+\tau)-\check{g}(v)\|\, dv\Bigr)/\lambda\Biggr)
\\ & \leq \varphi_{p(x)}\Biggl( \sum_{k=0}^{\infty} a_{k}\phi \Bigl(\int^{-x+(k+1)l}_{
-x+kl}a_{k}^{-1}\| R(v+x) \| \|\check{g}(v+\tau)-\check{g}(v)\|\, dv\Bigr)/\lambda\Biggr)
\\ & \leq \varphi_{p(x)}\Biggl( \sum_{k=0}^{\infty} a_{k}\phi \Bigl( l a_{k}^{-1}\cdot l^{-1}\int^{-x+(k+1)l}_{
-x+kl}\| R(v+x) \| \|\check{g}(v+\tau)-\check{g}(v)\|\, dv\Bigr) /\lambda\Biggr)
\\& \leq \varphi_{p(x)}\Biggl( l^{-1} \sum_{k=0}^{\infty}a_{k}\varphi(la_{k}^{-1}) \int^{-x+(k+1)l}_{
-x+kl}\phi \Bigl( \| R(v+x) \| \|\check{g}(v+\tau)-\check{g}(v)\|\Bigr) \, dv/\lambda\Biggr)
\\& \leq \varphi_{p(x)}\Biggl(  l^{-1} \sum_{k=0}^{\infty}a_{k}\varphi(la_{k}^{-1})  \int^{-x+(k+1)l}_{
-x+kl} \varphi\bigl(\| R(v+x) \| \bigr) \phi \Bigl( \| \check{g}(v+\tau)-\check{g}(v)\|\Bigr) \, dv/\lambda\Biggr)
\\& \leq \varphi_{p(x)}\Biggl(2  l^{-1} \sum_{k=0}^{\infty}a_{k}\varphi(la_{k}^{-1}) \bigl\| \varphi\bigl(\| R(v+x) \| \bigr)\bigr\|_{L^{q(v)}[-x+kl,-x+(k+1)l]} 
\\ & \times \phi \Bigl( \bigl \| \check{g}(v+\tau)-\check{g}(v)\bigr \|\Bigr)_{L^{p(v)}[-x+kl,-x+(k+1)l]}
\Bigr)/\lambda\Biggr)
\\ \leq & \varphi_{p(x)}\Biggl(2  l^{-1} \sum_{k=0}^{\infty}a_{k}\varphi(la_{k}^{-1})  \bigl\| \varphi\bigl(\| R(v+x) \| \bigr)\bigr\|_{L^{q(v)}[-x+kl,-x+(k+1)l]} 
\epsilon F(l,-x+kl)^{-1}/\lambda\Biggr).
\end{align*}
Let $K\subseteq {\mathbb R}$ be an arbitrary compact set.
Since the above computation holds for every real number $\tau \in {\mathbb R}$ and for every arbitrarily large real number $l >0,$ we can find $t\in{\mathbb R}$ such that $K\subseteq [t,t+l].$
Now we get from \eqref{okje} that the function 
$ \phi(\|G(\cdot +\tau)-G(\cdot)\|)$ belongs to the space $L^{p(x)}(K)$ by definition.
Condition \eqref{okje} and the above computation also imply that for each real number $t\in {\mathbb R}$ we have
$$
\int^{t+l}_{t}\varphi_{p(x)}\Bigl(\phi(\|G(x +\tau)-G(x)\|) /\lambda \Bigr)\, dx \leq 1,
$$
with $\lambda =\epsilon F_{1}(l,t),$
which simply implies the final conclusion.
\end{proof}

\begin{rem}\label{krajeq}
\begin{itemize}
\item[(i)] Suppose that $p(x)\equiv p\in [1,\infty).$ Then condition \eqref{okje} can be weakened to
\begin{align}\label{okjeokje}
\int^{t+l}_{t}\varphi_{p(x)}\Bigl( l^{-1} H(l,x)F_{1}(l,t)^{-1} \Bigr)\, dx \leq 1,
\end{align}
resp. there exists $l_{0}>0$ such that for all $l\geq l_{0}$ and $t\in {\mathbb R}$ we have \eqref{okjeokje}.
\item[(ii)]
Suppose that $\phi(x)=\varphi(x)=\psi(x)=x.$ Then condition \eqref{okje}, resp. \eqref{okjeokje}, holds provided that $l\geq 1$ and the term in the large brackets in this equation does not exceed $1/l$ or that $0<l<1$ and 
the term in the large brackets in this equation does not exceed $1.$ Similar comments can be made in the case of consideration of Theorem \ref{jensenjen} below (see also Corollary \ref{prcko}).
\end{itemize}
\end{rem}

\begin{cor}\label{kraj}
Suppose that condition \emph{(A1)} holds  with $I={\mathbb R}$, $p(x)\equiv p\geq 1,$ $1/p +1/q=1,$ $\check{g} : {\mathbb R} \rightarrow X$ is (equi-)Weyl-$(p,\phi,F)$-almost periodic and measurable, $F_{1} : (0,\infty) \times I \rightarrow (0,\infty),$
$(R(t))_{t> 0}\subseteq L(X,Y)$ is a strongly continuous operator family and $(a_{k})$ is a sequence of positive real numbers such that $\sum_{k=0}^{\infty}a_{k}=1.$
If for every real numbers $x,\ \tau \in {\mathbb R}$ we have \eqref{picni}
and if,
for every $t\in {\mathbb R}$ and $l>0,$ we have 
\begin{align}\label{seljak}
H_{p}(l,x):=\sum_{k=0}^{\infty}a_{k}\varphi(la_{k}^{-1}) \bigl\| \varphi\bigl(\| R(\cdot) \| \bigr)\bigr\|_{L^{q}[kl,(k+1)l]} F(l,-x+lk)^{-1}<\infty
\end{align}
and
\begin{align}\label{okjeasad}
\int^{t+l}_{t}\Bigl( l^{-1} H_{p}(l,x)F_{1}(l,t)^{-1}\Bigr)^{p}\, dx \leq 1,
\end{align}
resp. if \eqref{seljak} holds and there exists $l_{0}>0$ such that for all $l\geq l_{0}$ and $t\in {\mathbb R}$ we have \eqref{okjeasad},
then the function $G: {\mathbb R} \rightarrow Y,$ given by \eqref{wer},
is well-defined and (equi-)Weyl-$(p,\phi,F_{1})$-almost periodic.
\end{cor}

Now we will state and prove the following result with regards to the class of (equi-)Weyl-$[p,\phi,F]$-almost periodic functions:  

\begin{thm}\label{jensenjen}
Suppose that condition \emph{(B1)} holds  with $I={\mathbb R}$, 
$g :{\mathbb R} \rightarrow X$ is measurable, $\omega  : (0,\infty) \rightarrow (0,\infty),$
$F : (0,\infty) \times I \rightarrow (0,\infty),$ $(a_{k})$ is a sequence of positive real numbers such that $\sum_{k=0}^{\infty}a_{k}=1,$ 
$(b_{k})_{k\geq 0}$ is a sequence of positive real numbers, $S : (0,\infty) \times {\mathbb R} \rightarrow (0,\infty)$ is a given function, as well as 
for each $\epsilon>0$ we can find two real numbers $l>0$ and $L>0$ such that any interval $I'\subseteq I$ of length $L$ contains a point $\tau \in  I'$ such that
\begin{align}\label{proficeti}
\sup_{x\in [0,1]} \Bigl[\phi \bigl( \bigl\| g(x l+t-r-k+\tau) -g(xl+t-r-k)\bigr\|\bigr)_{L^{p(r)}[0,1]}\Bigr]\leq \omega(\epsilon ) b_{k}S(l,t)
\end{align}
for  any integer $k\geq 0$ and real number $t\in {\mathbb R}.$
Suppose, further, that the second inequality in \eqref{picni} holds,
$p,\ q\in {\mathcal P}([0,1]),$ $1/p(x) +1/q(x)=1$
and $(R(t))_{t> 0}\subseteq L(X,Y)$ is a strongly continuous operator family.
If for every real numbers $t,\ \tau \in {\mathbb R},$ every positive real number $l>0$ and every real number $x\in [0,1]$ we have
\begin{align}\label{micni}
\int^{\infty}_{
0}\| R(v) \| \|g(xl+t+\tau -v)-g(xl+t-v)\|\, dv <\infty,
\end{align}
and if, for every $t\in {\mathbb R},$ $x\in [0,1]$ and $l,\ \epsilon>0,$ we have
\begin{align}\label{ocek}
W(x):=\sum_{k=0}^{\infty} a_{k}\varphi \bigl( a_{k}^{-1} \bigr) \bigl\| \varphi\bigl(\| R(v+x) \| \bigr)\bigr\|_{L^{q(v)}[0,1]} b_{k}<\infty,
\end{align}
\begin{align}\label{okjejen}
\int^{1}_{0}\varphi_{p(x)}\Bigl(2\epsilon^{-1}F_{1}(l,t)^{-1}\omega(\epsilon) S(l,t) W(x) \Bigr)\, dx \leq 1,
\end{align}
resp. if \eqref{ocek} holds and there exists $l_{0}>0$ such that for all $l\geq l_{0},$ $\epsilon>0$ and $t\in {\mathbb R}$ we have \eqref{okjejen},
then the function $G: {\mathbb R} \rightarrow Y,$ given by \eqref{wer},
is well-defined and (equi-)Weyl-$[p,\phi,F_{1}]$-almost periodic.
\end{thm}

\begin{proof}
We will prove the theorem only for the class of equi-Weyl-$[p,\phi,F]$-almost periodic functions. As above, the function $G(\cdot)$ is well-defined.
Let $\epsilon>0$ be a fixed real number. Then we can find two real numbers $l>0$ and $L>0$ such that any interval $I'\subseteq I$ of length $L$ contains a point $\tau \in  I'$ such that
\eqref{proficeti} holds for any integer $k\geq 0$ and any real number $t\in {\mathbb R}.$
Using our assumptions from condition (B1), the
Jensen integral inequality applied to the function $\phi(\cdot)$ (see also condition \eqref{micni}), the fact that the functions $\phi(\cdot)$ and $\varphi_{p(x)}(\cdot)$ are monotonically increasing, \eqref{infinitever} and Lemma \ref{aux}(i), we get that, for every real numbers $x\in [0,1]$ and  $t\in {\mathbb R},$ the following holds:
\begin{align*}
&\varphi_{p(x)}\Bigl(\phi(\|G(xl +t+\tau)-G(xl+t)\|) /\lambda \Bigr)
\\  \leq & \varphi_{p(x)}\Biggl( \phi \Bigl(\int^{\infty}_{
0}\| R(v) \| \|g(xl+t+\tau -v)-g(xl+t-v)\|\, dv\Bigr)/\lambda\Biggr)
\\= & \varphi_{p(x)}\Biggl(  \phi \Bigl( \sum_{k=0}^{\infty} a_{k}\int^{1}_{
0}a_{k}^{-1}\| R(v+k) \| \| g(xl+t+\tau-v-k)-g(xl+t-v-k)\| \, dv \Bigr) /\lambda\Biggr)
\\ \leq &  \varphi_{p(x)}\Biggl( \sum_{k=0}^{\infty} a_{k} \int^{1}_{
0}\phi \Bigl( a_{k}^{-1}\| R(v+k) \| \| g(xl+t+\tau-v-k)-g(xl+t-v-k)\| \, dv \Bigr) /\lambda\Biggr)
\\ \leq &  \varphi_{p(x)}\Biggl( \sum_{k=0}^{\infty} a_{k}\varphi \bigl(a_{k}^{-1} \bigr)  \int^{1}_{
0}\varphi \bigl(\| R(v+k) \|\bigr)
\\ \times & \phi \Bigl( \| g(xl+t+\tau-v-k)-g(xl+t-v-k)\|  \Bigr) \, dv/\lambda\Biggr)
\\ \leq & \varphi_{p(x)}\Biggl( \frac{2}{\lambda} \sum_{k=0}^{\infty}a_{k}\varphi \bigl(a_{k}^{-1} \bigr)   \varphi \bigl(\| R(v+k) \|\bigr)_{L^{q(v)}[0,1]}
\\ \times & \phi \Bigl( \| g(xl+t+\tau-v-k)-g(xl+t-v-k)\|  \Bigr)_{L^{p(v)}[0,1]} \Biggr)
\\  \leq & \varphi_{p(x)}\Biggl( \frac{2}{\lambda} \sum_{k=0}^{\infty} a_{k}\varphi \bigl(a_{k}^{-1} \bigr)  \varphi \bigl(\| R(v+k) \|\bigr)_{L^{q(v)}[0,1]}\omega(\epsilon) b_{k}S(l,t)\Biggr).
\end{align*}
Arguing as in the proof of Theorem \ref{jensen}, we get from condition \eqref{okjejen} that the function $\phi ( \| G(\cdot  l+t+\tau) -G(t+\cdot l)\|)$ belongs to the space $ L^{p(\cdot)}([0,1])$
for arbitrary real numbers $\tau ,\ t\in {\mathbb R}$ and $l>0.$
Condition \eqref{okjejen} implies that for each real numbers $t\in {\mathbb R}$ and $x\in [0,1]$ we have
$$
\int^{1}_{0}\varphi_{p(x)}\Bigl(\phi(\|G(xl +t+\tau)-G(xl+t)\|) /\lambda \Bigr)\, dx \leq 1,
$$
with $\lambda =\epsilon F_{1}(l,t)^{-1},$
which simply implies the final conclusion.
\end{proof}

\begin{cor}\label{prcko}
Suppose that condition \emph{(B1)} holds with $I={\mathbb R}$ and $p(x)\equiv p\in [1,\infty),$ $1/p+1/q=1,$ 
$g :{\mathbb R} \rightarrow X$ is measurable, $\omega  : (0,\infty) \rightarrow (0,\infty),$
$F : (0,\infty) \times I \rightarrow (0,\infty),$ $(a_{k})$ is a sequence of positive real numbers such that $\sum_{k=0}^{\infty}a_{k}=1,$
$(b_{k})_{k\geq 0}$ is a sequence of positive real numbers, $S : (0,\infty) \times {\mathbb R} \rightarrow (0,\infty)$ is a given function, as well as 
for each $\epsilon>0$ we can find real numbers $l>0$ and $L>0$ such that any interval $I'\subseteq I$ of length $L$ contains a point $\tau \in  I'$ such that 
\eqref{proficeti} holds with $p(r)\equiv p,$
for  any integer $k\geq 0$ and any real number $t\in {\mathbb R}.$
Suppose, further, that the second inequality in \eqref{picni} holds,
and $(R(t))_{t> 0}\subseteq L(X,Y)$ is a strongly continuous operator family.
If for every real numbers $t,\ \tau \in {\mathbb R},$ every positive real number $l>0$ and every real number $x\in [0,1]$ we have
\eqref{micni},
and if, for every $t\in {\mathbb R},$ $x\in [0,1]$ and $l>0,$ we have
\begin{align}\label{oceksine}
W_{p}(x):=\sum_{k=0}^{\infty} a_{k}\varphi \bigl( a_{k}^{-1} \bigr) \bigl\| \varphi\bigl(\| R(\cdot) \| \bigr)\bigr\|_{L^{q}[x,x+1]} b_{k}<\infty
\end{align}
and
\begin{align}\label{okjejensine}
\int^{1}_{0}\varphi_{p(x)}\Bigl(2F_{1}(l,t)^{-1}S(l,t) W_{p}(x) \Bigr)\, dx \leq 1,
\end{align}
resp. if \eqref{oceksine} holds and there exists $l_{0}>0$ such that for all $l\geq l_{0}$ and $t\in {\mathbb R}$ we have \eqref{okjejensine},
then the function $G: {\mathbb R} \rightarrow Y,$ given by \eqref{wer},
is well-defined and (equi-)Weyl-$[p,\phi,F_{1}]$-almost periodic.
\end{cor}

Concerning Theorem \ref{jensenjen}, it should be noted that, in \cite[Proposition 6.1]{AP},  we have analyzed the situation in which the function $\check{g} : {\mathbb R} \rightarrow X$
is $S^{p(x)}$-almost periodic and $\sum_{k=0}^{\infty}\| R(\cdot +k)\|_{L^{q(\cdot)}[0,1]}<\infty.$ Then the resulting function $G(\cdot)$ is almost periodic, which cannot be derived from the above-mentioned theorem.

For the class of (equi-)Weyl-$(p,\phi,F)_{1}$-almost periodic functions, we will state the following result:

\begin{thm}\label{jensen-12121121}
Suppose that $\check{g} : {\mathbb R} \rightarrow X$ is (equi-)Weyl-$(p,\phi,F)_{1}$-almost periodic and measurable, $F_{1} : (0,\infty) \times I \rightarrow (0,\infty),$ $p,\ q\in {\mathcal P}({\mathbb R}),$ $1/p(x) +1/q(x)=1,$
$(R(t))_{t> 0}\subseteq L(X,Y)$ is a strongly continuous operator family and
for every real numbers $x,\ \tau \in {\mathbb R}$ we have \eqref{picni}. Suppose that, for every real number $t\in {\mathbb R}$ and
positive real numbers $l,\ \epsilon>0,$ there exist two positive real numbers $a>0$ and $\lambda >0$ such that $\lambda \leq a,$ $[0,a]\subseteq \phi^{-1}([0,\epsilon F(l,t)^{-1}]),$
\begin{align}\label{radio121}
\sum_{k=0}^{\infty}\bigl\| R(v+x) \bigr\|_{L^{q(v)}[-x+kl,-x+(k+1)l]}\sup \phi^{-1}\Bigl(\bigl[0,\epsilon F(l,-x+kl)^{-1}\bigr]\Bigr)<\infty
\end{align}
and
\begin{align}\label{radio121123}
\int^{t+l}_{t}\varphi_{p(x)}\Biggl(2\frac{\sum_{k=0}^{\infty}\bigl\| R(v+x) \bigr\|_{L^{q(v)}[-x+kl,-x+(k+1)l]}\sup \phi^{-1}\Bigl(\bigl[0,\epsilon F(l,-x+kl)^{-1}\bigr]\Bigr)}{\lambda} \Biggr)\, dx\leq 1,
\end{align}
resp. \eqref{radio121} holds and there exists $l_{0}>0$ such that for all $l\geq l_{0},$ $\epsilon>0$ and $t\in {\mathbb R}$ we have \eqref{radio121123},
then the function $G: {\mathbb R} \rightarrow Y,$ given by \eqref{wer},
is well-defined and (equi-)Weyl-$(p,\phi,F_{1})_{1}$-almost periodic.
\end{thm}

\begin{proof}
As in the proof of Theorem \ref{jensen}, we have that 
the function $G(\cdot)$ is well-defined and the integrals in definitions of $G(x)$ and $G(x+\tau)-G(x)$ converge absolutely ($ x,\ \tau\in {\mathbb R}$). By Lemma \ref{aux}(ii), we get that the function $G(\cdot+\tau)-G(\cdot)$ belongs to the space $L^{p(x)}(K)$ for each compact set $K\subseteq {\mathbb R}.$ The remainder follows similarly as in the proof of Theorem \ref{jensen}, by using condition \eqref{radio121}, as well as the estimates
\begin{align*}
\|G(x+\tau)& -G(x)\|\leq 2\sum_{k=0}^{\infty}\bigl\| R(v+x) \bigr\|_{L^{q(v)}[-x+kl,-x+(k+1)l]}
\\ & \times \bigl\| \check{g}(v+\tau)-\check{g}(v) \bigr\|_{L^{p(v)}[-x+kl,-x+(k+1)l]}
\end{align*}
and
\begin{align*}
 \bigl\| \check{g}(v+\tau)-\check{g}(v) \bigr\|_{L^{p(v)}[-x+kl,-x+(k+1)l]} \leq \sup \phi^{-1}\Bigl(\bigl[0,\epsilon F(l,-x+kl)^{-1}\bigr]\Bigr),
\end{align*}
and the equivalence relation
\begin{align*}
\phi\Bigl(\bigl\|G(\cdot &+\tau)-G(\cdot)\bigr\|_{L^{p(x)}[t,t+l]} \Bigr)\leq \epsilon F_{1}(l,t)^{-1} 
\\ & \Leftrightarrow  
\bigl\|G(\cdot+\tau)-G(\cdot)\bigr\|_{L^{p(x)}[t,t+l]} \leq \phi^{-1}\Bigl( \bigl[0,\epsilon F_{1}(l,t)^{-1}\bigr]\Bigr),
\end{align*}
for any $x,\ t,\ \tau \in {\mathbb R}$ and $l>0.$
\end{proof}

Concerning the class of (equi-)Weyl-$[p,\phi,F]_{1}$-almost periodic functions, we can state the following result; the proof can be deduced as above and therefore omitted (we can similarly formulate analogues of Corollary \ref{kraj} and Corollary \ref{prcko}, as well as the conclusions from Remark \ref{krajeq}):

\begin{thm}\label{napolje-vani}
Suppose that 
$g :{\mathbb R} \rightarrow X$ is measurable, $\omega  : (0,\infty) \rightarrow (0,\infty),$
$F : (0,\infty) \times I \rightarrow (0,\infty),$ 
$(b_{k})_{k\geq 0}$ is a sequence of positive real numbers, $S : (0,\infty) \times {\mathbb R} \rightarrow (0,\infty)$ is a given function, as well as 
for each $\epsilon>0$ we can find two real numbers $l>0$ and $L>0$ such that any interval $I'\subseteq I$ of length $L$ contains a point $\tau \in  I'$ such that 
\begin{align*}
\sup_{x\in [0,1]} \Bigl[\bigl\| g(x l+t-r-k+\tau) -g(xl+t-r-k)\bigr\|_{L^{p(r)}[0,1]}\Bigr]\leq \omega(\epsilon ) b_{k}S(l,t)
\end{align*}
for  any integer $k\geq 0$ and real number $t\in {\mathbb R}.$
Suppose, further, that the second inequality in \eqref{picni} holds,
$p,\ q\in {\mathcal P}([0,1]),$ $1/p(x) +1/q(x)=1$
and $(R(t))_{t> 0}\subseteq L(X,Y)$ is a strongly continuous operator family.
If for every real numbers $t,\ \tau \in {\mathbb R},$ every positive real number $l>0$ and every real number $x\in [0,1]$ we have \eqref{micni},
if
\begin{align}\label{ocek121}
W_{2}(x):=\sum_{k=0}^{\infty} \bigl\| R(v+x) \bigr\|_{L^{q(v)}[0,1]} b_{k}<\infty,\quad x\in [0,1],
\end{align}
and if, for every $t\in {\mathbb R}$ and $l,\ \epsilon>0,$ we have the existence of two positive real numbers $a>0$ and $\lambda >0$ such that $\lambda \leq a,$ $[0,a]\subseteq \phi^{-1}([0,\epsilon F_{1}(l,t)^{-1}])$ and 
\begin{align}\label{okee}
\int^{1}_{0}\varphi_{p(x)}\Biggl(2 \frac{\omega(\epsilon) S(l,t) W_{2}(x)}{\lambda} \Biggr)\, dx \leq 1,
\end{align}
resp. if \eqref{ocek121} holds and there exists $l_{0}>0$ such that for all $l\geq l_{0},$ $\epsilon>0$ and $t\in {\mathbb R}$ we have \eqref{okee},
then the function $G: {\mathbb R} \rightarrow Y,$ given by \eqref{wer},
is well-defined and (equi-)Weyl-$[p,\phi,F_{1}]$-almost periodic.
\end{thm}

\begin{rem}\label{katastr}
The assertions of Theorem \ref{jensen-12121121}, resp. Theorem \ref{napolje-vani}, can be much simpler formulated provided that:
\begin{itemize}
\item[(A2):] The function $\phi : [0,\infty) \rightarrow [0,\infty)$ is a monotonically increasing bijection and $p\in {\mathcal P}({\mathbb R}),$ resp.
\item[(B2):] The function $\phi : [0,\infty) \rightarrow [0,\infty)$ is a monotonically increasing bijection and $p\in {\mathcal P}([0,1]).$
\end{itemize}
Any of these conditions implies that the function $\phi^{-1} : [0,\infty) \rightarrow [0,\infty)$ is a monotonically increasing bijection, as well.
If condition (A2), resp. (B2), holds, then the class of (equi-)Weyl-$(p,\phi,F)_{2}$-almost periodic functions, resp. (equi-)Weyl-$[p,\phi,F]_{2}$-almost periodic functions, 
coincides with the class of (equi-)Weyl-$(p,x,F)_{2}$-almost periodic functions, resp. (equi-)Weyl-$[p,x,F]_{2}$-almost periodic functions. 
\end{rem}

Concerning the invariance of 
(equi-)Weyl-$(p,\phi,F)_{2}$-almost periodicity and (equi-)Weyl-$[p,\phi,F]_{2}$-almost periodicity under the actions of infinite convolution products, we will only state the following analogues of Theorem \ref{jensen-12121121} and Theorem \ref{napolje-vani}:

\begin{thm}\label{jensen-121211212}
Suppose that $\check{g} : {\mathbb R} \rightarrow X$ is (equi-)Weyl-$(p,\phi,F)_{2}$-almost periodic and measurable, $F_{1} : (0,\infty) \times I \rightarrow (0,\infty),$ $p,\ q\in {\mathcal P}({\mathbb R}),$ $1/p(x) +1/q(x)=1,$
$(R(t))_{t> 0}\subseteq L(X,Y)$ is a strongly continuous operator family and
for every real numbers $x,\ \tau \in {\mathbb R}$ we have \eqref{picni}. Suppose that, for every real number $t\in {\mathbb R}$ and
positive real numbers $l,\ \epsilon>0,$ there exist two positive real numbers $a>0$ and $\lambda >0$ such that $\lambda \leq a,$ $[0,a]\subseteq F(l,t)^{-1}\phi^{-1}([0,\epsilon ]),$ 
\begin{align}\label{radio121okee}
\sum_{k=0}^{\infty}\bigl\| R(v+x) \bigr\|_{L^{q(v)}[-x+kl,-x+(k+1)l]}F(l,-x+kl)^{-1}<\infty
\end{align}
and
\begin{align}\label{radio1211212345}
\int^{t+l}_{t}\varphi_{p(x)}\Biggl(2\frac{\sum_{k=0}^{\infty}\bigl\| R(v+x) \bigr\|_{L^{q(v)}[-x+kl,-x+(k+1)l]}F(l,-x+kl)^{-1}\sup   \phi^{-1}\bigl([0,\epsilon]\bigr)}{\lambda} \Biggr)\, dx\leq 1,
\end{align}
resp. \eqref{radio121okee} holds and there exists $l_{0}>0$ such that for all $l\geq l_{0},$ $\epsilon>0$ and $t\in {\mathbb R}$ we have \eqref{radio1211212345},
then the function $G: {\mathbb R} \rightarrow Y,$ given by \eqref{wer},
is well-defined and (equi-)Weyl-$(p,\phi,F_{1})_{2}$-almost periodic.
\end{thm}

\begin{thm}\label{napolje-vani-univer}
Suppose that, with the exception of equation \eqref{okee}, all remaining assumptions from the formulation of \emph{Theorem \ref{napolje-vani}} hold.
If for every $t\in {\mathbb R}$ and $l,\ \epsilon>0$ we have the existence of two positive real numbers $a>0$ and $\lambda >0$ such that $\lambda \leq a,$ $[0,a]\subseteq F_{1}(l,t)^{-1}\phi^{-1}([0,\epsilon ])$ and 
\begin{align}\label{okjejen121}
\int^{1}_{0}\varphi_{p(x)}\Biggl(2 \frac{\omega(\epsilon) S(l,t) W_{2}(x)}{\lambda} \Biggr)\, dx \leq 1,
\end{align}
resp. if \eqref{ocek121} holds and there exists $l_{0}>0$ such that for all $l\geq l_{0},$ $\epsilon>0$ and $t\in {\mathbb R}$ we have \eqref{okjejen121},
then the function $G: {\mathbb R} \rightarrow Y,$ given by \eqref{wer},
is well-defined and (equi-)Weyl-$[p,\phi,F_{2}]$-almost periodic.
\end{thm}

The invariance of asymptotical Weyl-$p$-almost periodicity under the action of finite convolution product, where the exponent $p\in [1,\infty)$ has a constant value, has been examined in \cite[Proposition 5.3, Example 5.4-5.6]{weyl} and \cite[Proposition 1, Remark 2-Remark 5]{fedorov-novi}.
Concerning the invariance of asymptotical Weyl-$p(x)$-almost periodicity under the action of finite convolution product, we will state and prove only one proposition (see \cite[Proposition 6.3]{AP} for a corresponding result regarding asymptotical Stepanov-$p(x)$-almost periodicity). In order to do that, suppose that (see also Definition \ref{asym-weyl-variable}, where the domain of function $g(\cdot)$ is the non-negative real axis)
there exist two functions 
$g : {\mathbb R} \rightarrow X$ and $q : [0,\infty) \rightarrow X$
such that $h(t)=g(t)+q(t)$ for a.e. $t\geq 0,$ $g(\cdot)$ belongs to some of function spaces introduced in Definition \ref{w-orlicz}-Definition \ref{w-orliczaq2} or Definition \ref{w-orliczti}-Definition \ref{w-orlicztistr}, with $I={\mathbb R},$
and $q(\cdot)$ belongs to some of function spaces introduced in Definition \ref{stea-weyl0px}-Definition \ref{stea-weyl0px2}, with $I=[0,\infty).$ The study of qualitative properties of the function
\begin{align*}
t\mapsto H(t)\equiv \int^{t}_{0}R(t-s)[g(s)+q(s)]\, ds,\quad t\geq 0
\end{align*}
is based on the decomposition
$$
H(t)=\int^{t}_{0}R(t-s)q(s)\, ds+\Biggl[ \int^{t}_{-\infty}R(t-s)g(s)\, ds -\int^{\infty}_{t}R(s)g(t-s)\, ds \Biggr],\quad t\geq 0
$$
and the use of corresponding results for infinite convolution product. In the following proposition, we will consider the qualitative properties of functions
\begin{align}\label{bit-zero}
t\mapsto H_{1}(t)\equiv\int^{\infty}_{t}R(s)g(t-s)\, ds,\quad t\geq 0
\end{align}
and 
\begin{align}\label{bit}
t\mapsto H_{2}(t)\equiv \int^{t}_{0}R(t-s)q(s)\, ds,\quad t\geq 0
\end{align}
separately. In the first part of proposition, we continue our analysis from \cite[Proposition 5.2]{toka-marek-prim}; our previous results show that the case $p(x)\equiv p>1$ is not simple in the analysis of asymptotical Weyl-$p$-almost periodicity so that we will consider the case $p(x)\equiv 1$ in the second part, with the notion introduced in Definition \ref{stea-weyl0px}(i) only (cf. also \cite[Proposition 5.3(i)]{weyl}).

\begin{prop}\label{finite}
\begin{itemize}
\item[(i)] Suppose that $p,\ q\in {\mathcal P}([0,1]),$ $1/p(x) +1/q(x)=1$
and $(R(t))_{t> 0}\subseteq L(X,Y)$ is a strongly continuous operator family. Let the function $\check{g} : {\mathbb R} \rightarrow X$ be Stepanov $p(x)$-bounded and let for each $t\geq 0$ the series $\sum_{k=0}^{\infty}\| R(\cdot +t+k)\|_{L^{q(\cdot)}[0,1]}\equiv S(t)$ be convergent. Then the function $H_{1}(\cdot),$ given by \eqref{bit-zero}, is well-defined. Furthermore, this function is continuous provided that the Bochner transform $\hat{\check g} : {\mathbb R} \rightarrow L^{p(x)}([0,1])$ is uniformly continuous, while the function
$H_{1}(\cdot)$ satisfies $\lim_{t\rightarrow +\infty}H_{1}(t)=0$ provided that $\lim_{t\rightarrow +\infty}S(t)=0.$
\item[(ii)] Suppose that 
$(R(t))_{t> 0}\subseteq L(X,Y)$ is a strongly continuous operator family such that $\int^{\infty}_{0}\|R(s)\|_{L(X,Y)}\, ds <\infty .$ Let the function $q : [0,\infty) \rightarrow Y$ be equi-Weyl-$(1,x,F)$-vanishing and let $F_{1} : (0,\infty) \times [0,\infty) \rightarrow (0,\infty).$ If
for each $\epsilon>0$ there exists $l_{0}>0$ such that for each $l>l_{0}$ there exists $t_{0,l}>0$ such that for each $t\geq t_{0,l}$ we have 
$$
\sup_{x\geq 0}\Biggl [F_{1}(l,t)\int^{x+t}_{0} \Biggr[\int_{x+t}^{x+t+l}\bigl \| R(s-r)\bigr\|_{L(X,Y)}\, ds \Biggr ]\bigl \| q(r)\bigr\|_{Y}\, dr  \Biggr]<\epsilon,
$$
and if, additionally, there exists a finite constant $M>0$ such that
\begin{align}\label{zran}
\frac{F_{1}(l,t)}{F(l,t)}\leq M,\quad l>0,\ t\geq 0,
\end{align}
then the function $H_{2}(\cdot)$ is equi-Weyl-$(1,x,F_{1})$-vanishing.
\end{itemize}
\end{prop}

\begin{proof}
(i): The first part follows from the Stepanov $p(x)$-boundedness of function $\check{g}(\cdot)$ and the following simple  computation
\begin{align*}
\Biggl\| \int^{\infty}_{t}& R(s)\check{g}(s-t)\, ds \Biggr \|=\Biggl\| \sum_{k=0}^{\infty} \int^{1}_{0} R(s+t+k)\check{g}(s+k)\, ds \Biggr \|
\\ & \leq 2\sum_{k=0}^{\infty}\bigl\| R(\cdot +t+k)\bigr\|_{L^{q(\cdot)}[0,1]}\sup_{k\in {\mathbb N}_{0}}\bigl\| \check{g}(\cdot +k) \bigr\|_{L^{p(\cdot)}[0,1]}.
\end{align*}
This computation also shows that $\lim_{t\rightarrow +\infty}H_{1}(t)=0$ provided that $\lim_{t\rightarrow +\infty}S(t)=0.$ For remainder, let us suppose that the function $\hat{\check g} : {\mathbb R} \rightarrow L^{p(x)}([0,1])$ is uniformly continuous.
Let $(t_{n})$ be a sequence of non-negative reals converging to a number $t\geq 0.$ Then we can use the  H\"older inequality and the decomposition
\begin{align*}
\int^{\infty}_{t}R(s)&g(t-s)\, ds-\int^{\infty}_{t_{n}}R(s)g\bigl(t_{n}-s\bigr)\, ds
\\ & =\int^{\infty}_{t}R(s)\bigl[\check{g}(s-t)-\check{g}\bigl(s-t_{n}\bigr)\bigr]\, ds +\int^{t_{n}}_{t}R(s)\check{g}(s-t)\, ds,\quad n\in {\mathbb N}
\end{align*}
in order to see that
\begin{align*}
\Biggl\| & \int^{\infty}_{t}R(s)g(t-s)\, ds-\int^{\infty}_{t_{n}}R(s)g\bigl(t_{n}-s\bigr)\, ds\Biggr\|
\\ & \leq 2\sum_{k=0}^{\infty}\| R(\cdot +t+k)\|_{L^{q(\cdot)}[0,1]}\sup_{k\in {\mathbb N}_{0}}\bigl\| \check{g}(\cdot +k) - \check{g}(\cdot +k+(t-t_{n}))\bigr\|_{L^{p(\cdot)}[0,1]}
\\ & + 2\bigl \| R(\cdot)\bigr\|_{L^{q(\cdot)}[0,|t_{n}-t|]}\bigl \| \check{g}(\cdot)\bigr\|_{L^{p(\cdot)}[0,1]},\quad n\in {\mathbb N}.
\end{align*}
Since $\bigl \| R(\cdot)\bigr\|_{L^{q(\cdot)}[0,|t_{n}-t|]} \rightarrow 0$ as $n\rightarrow +\infty$ (see, e.g., \cite[Lemma 3.2.8(c)]{variable}) and the function $\hat{\check g} : {\mathbb R} \rightarrow L^{p(x)}([0,1])$ is uniformly continuous, the proof of the first part is completed.

(ii): By the proof of \cite[Proposition 5.3(i)]{weyl}, we have 
\begin{align*}
F_{1}(l,t)\int^{x+t+l}_{x+t}&\bigl\|H_{2}(s)\bigr\|_{Y}\, ds \leq F_{1}(l,t)\int^{x+t}_{0} \Biggr[\int_{x+t}^{x+t+l}\bigl \| R(s-r)\bigr\|_{L(X,Y)}\, ds \Biggr ]\bigl \| q(r)\bigr\|_{Y}\, dr
\\ & +F_{1}(l,t)
\Biggl[\int_{0}^{\infty}\bigl \| R(v)\bigr\|_{L(X,Y)}\, dv \Biggr] \ \cdot \ \int^{x+t+l}_{x+t}\bigl \| q(r)\bigr\|_{Y}\, dr,
\end{align*}
for any $x\geq 0$ and $ l>0.$ Our preassumption shows that the first addend is equi-Weyl-$(1,x,F_{1})$-vanishing. The second addend is likewise equi-Weyl-$(1,x,F_{1})$-vanishing because we have assumed that the function $q(\cdot)$ is equi-Weyl-$(1,x,F)$-vanishing
and condition \eqref{zran}.
\end{proof}

We round off this section by examing the convolution invariance of Weyl almost periodic functions with variable exponent. In order to do that, we shall basically follow the method proposed in the proof of Theorem \ref{jensen}.

\begin{prop}\label{zas-heban}
Suppose that $I={\mathbb R},$ $\psi \in L^{1}({\mathbb R}),$ $(a_{k})_{k\in {\mathbb Z}}$ is a sequence of positive real numbers satisfying $\sum_{k\in {\mathbb Z}}a_{k}=1$ and condition \emph{(A1)} holds true. Let $f\in (e-)W_{ap}^{(p,\phi,F)}({\mathbb R} : X) \cap L^{\infty}({\mathbb R} : X).$ Then the function 
\begin{align}\label{wounds}
x\mapsto (\psi \ast f)(x):= \int^{+\infty}_{-\infty}\psi(x-y)f(y)\, dy,\quad x\in {\mathbb R}
\end{align}
is well-defined and belongs to the space $L^{\infty}({\mathbb R} : X).$ Furthermore, if $p_{1}\in {\mathcal P}({\mathbb R}),$ $F_{1} : (0,\infty)  \times {\mathbb R} \rightarrow (0,\infty)$ and if, for every $t\in {\mathbb R}$ and $l>0,$ we have
\begin{align}\label{proba}
\int^{t+l}_{t}\varphi_{p_{1}(x)}\Biggl( 2l^{-1}F_{1}(l,t)\varphi(l)\sum_{k\in {\mathbb Z}}\frac{a_{k}\bigl\| \varphi \bigl( a_{k}^{-1}\psi(x-z) \bigr)\bigr\|_{L^{q(z)}[x-(k+1)l,x-kl]}}{F(l,x-(k+1)l)}\Biggr)\, dx \leq 1,
\end{align}
then we have $\psi \ast f \in (e-)W_{ap}^{(p_{1},\phi,F_{1})}({\mathbb R} : X).$
\end{prop}

\begin{proof}
The proof can be deduced by using the arguments contained in the proof of Theorem \ref{jensen}, the equality
\begin{align*}
\Bigl\| & \phi \bigl( \| (\psi \ast f)(\cdot +\tau)-(\psi \ast f)(\cdot) \|\bigr)\Bigr\|_{L^{p_{1}(\cdot)}[t,t+l]}
\\& =\inf\Biggl\{ \lambda >0 : \int^{t+l}_{t}\varphi_{p_{1}(x)}\Biggl( \frac{ \phi\bigl(\| (\psi \ast f)(x +\tau)-(\psi \ast f)(x)\| \bigr)}{\lambda}\Biggr)\, dx \leq 1\Biggr\} 
\\ & =\inf\Biggl\{ \lambda >0 : \int^{t+l}_{t}\varphi_{p_{1}(x)}\Biggl( \frac{\phi \bigl( \| \int^{+\infty}_{-\infty}\psi (y)[ f(x+\tau-y)-f(x-y)]\, dy\| \bigr)}{\lambda}\Biggr)\, dx \leq 1\Biggr\} 
\end{align*}
 and the following computation:
\begin{align*}
& \int^{t+l}_{t}\varphi_{p_{1}(x)}\Biggl( \frac{\phi \bigl( \| \int^{+\infty}_{-\infty}\psi (y)[ f(x+\tau-y)-f(x-y)]\, dy\| \bigr)}{\lambda}\Biggr)\, dx
\\ & \leq \int^{t+l}_{t}\varphi_{p_{1}(x)}\Biggl( \frac{\phi \bigl( \sum_{k\in {\mathbb Z}}a_{k}\| \int^{(k+1)l}_{kl}a_{k}^{-1}\psi (y)[ f(x+\tau-y)-f(x-y)]\, dy\| \bigr)}{\lambda}\Biggr)\, dx
\\ & \leq \int^{t+l}_{t}\varphi_{p_{1}(x)}\Biggl( \frac{\sum_{k\in {\mathbb Z}}a_{k}\phi \bigl( l^{-1}l\| \int^{(k+1)l}_{kl}a_{k}^{-1}\psi (y)[ f(x+\tau-y)-f(x-y)]\, dy\| \bigr)}{\lambda}\Biggr)\, dx
\\ & \leq \int^{t+l}_{t}\varphi_{p_{1}(x)}\Biggl( \frac{\sum_{k\in {\mathbb Z}}a_{k}\varphi(l)l^{-1}\int^{(k+1)l}_{kl}\phi \bigl( a_{k}^{-1}\psi (y)\|  f(x+\tau-y)-f(x-y)\|\bigr)\, dy }{\lambda}\Biggr)\, dx
\\ & \leq \int^{t+l}_{t}\varphi_{p_{1}(x)}\Biggl( \frac{\sum_{k\in {\mathbb Z}}a_{k}\varphi(l)l^{-1}\int^{(k+1)l}_{kl}\varphi \bigl( a_{k}^{-1}\psi (y) \bigr) \phi \bigl(\|  f(x+\tau-y)-f(x-y)\|\bigr)\, dy }{\lambda}\Biggr)\, dx
\\ & = \int^{t+l}_{t}\varphi_{p_{1}(x)}\Biggl( \frac{\sum_{k\in {\mathbb Z}}a_{k}\varphi(l)l^{-1}\int_{x-(k+1)l}^{x-kl}\varphi \bigl( a_{k}^{-1}\psi (x-z) \bigr) \phi \bigl(\|  f(z+\tau)-f(z)\|\bigr)\, dz }{\lambda}\Biggr)\, dx
\\ & \leq \int^{t+l}_{t}\varphi_{p_{1}(x)}\Biggl(2 \sum_{k\in {\mathbb Z}}a_{k}\varphi(l)l^{-1}\bigl\|\varphi \bigl( a_{k}^{-1}\psi (x-z) \bigr)\bigr\|_{L^{q(z)}[x-(k+1)l,x-kl]}
\\ & \times \frac{\bigl\|\phi \bigl(\|  f(z+\tau)-f(z)\|\bigr)\bigr\|_{L^{p(z)}[x-(k+1)l,x-kl]} }{\lambda}\Biggr)\, dx,
\end{align*}
which is valid for every $t,\ \tau \in {\mathbb R}$ and $l>0.$
\end{proof}

We can similarly prove the following result for the class of (equi-)Weyl-$[p,\phi,F]$-almost periodic functions:

\begin{prop}\label{zas-lena}
Suppose that $I={\mathbb R},$ $\psi \in L^{1}({\mathbb R}),$ $(a_{k})_{k\in {\mathbb Z}}$ is a sequence of positive real numbers satisfying $\sum_{k\in {\mathbb Z}}a_{k}=1$ and condition \emph{(B1)} holds true. Let $f\in (e-)W_{ap}^{[p,\phi,F]}({\mathbb R} : X) \cap L^{\infty}({\mathbb R} : X).$ Then the function $(\psi \ast f)(\cdot)$ defined by \eqref{wounds}
belongs to the space $L^{\infty}({\mathbb R} : X).$ Furthermore, if $p_{1}\in {\mathcal P}([0,1]),$ $F_{1} : (0,\infty)  \times {\mathbb R} \rightarrow (0,\infty)$ and if, for every $t\in {\mathbb R}$ and $l>0,$ we have
\begin{align}\label{proba-est}
\int^{1}_{0}\varphi_{p_{1}(x)}\Biggl( 2F_{1}(l,t)\sum_{k\in {\mathbb Z}}\frac{\bigl\| \varphi \bigl( la_{k}^{-1}\psi(xl-(z+k)l) \bigr)\bigr\|_{L^{q(z)}[0,1]}}{F(l,t+kl)}\Biggr)\, dx \leq 1,
\end{align}
then we have $\psi \ast f \in (e-)W_{ap}^{[p_{1},\phi,F_{1}]}({\mathbb R} : X).$
\end{prop}

In the case of consideration of constant coefficients, the coefficient $2$ in the equations \eqref{proba} and \eqref{proba-est} can be neglected. The interested reader may try to formulate the corresponding results for the classes of (equi-)Weyl-$(p,\phi,F)_{i}$-almost periodic functions and (equi-)Weyl-$[p,\phi,F]_{i}$-almost periodic functions, where $i=1,2,$ as well as to formulate an extension of \cite[Proposition 4.3]{weyl} for Weyl almost periodic functions with variable exponent.

\subsection{Growth order of $(R(t))_{t>0}$}\label{rates}

In this subsection, we will analyze solution operator families $(R(t))_{t>0}\subseteq L(X,Y)$ which satisfies condition 
\begin{align}\label{isti-ee}
\|R(t)\|_{L(X,Y)}\leq \frac{Mt^{\beta -1}}{1+t^{\gamma}},\ t>0\mbox{ for some finite constants }\gamma>1,\ \beta \in (0,1],\ M>0,
\end{align}
or
condition
\begin{align}\label{rajasad}
\| R(t)\|_{L(X,Y)}\leq Mt^{\beta-1}e^{-ct},\quad t>0 \mbox{ for some finite constants } \beta\in (0,1] \mbox{ and } c>0.
\end{align}

For simplicity, we will analyze only the constant exponents $p(x)\equiv p \in [1,\infty)$ as well the class of (equi-)Weyl-$(p,\phi,F)$-almost periodic functions and the class of (equi-)Weyl-$(p,\phi,F)_{i}$-almost periodic functions, where $i=1,2.$ So, let $1/p+1/q=1$ and let $(R(t))_{t>0}\subseteq L(X,Y)$ satisfy 
\eqref{isti-ee} or \eqref{rajasad}. We will additionally assume that $q(\beta-1)>-1$ provided that $p>1$, resp. $\beta=1,$ provided that $p=1$.

In \cite[Proposition 2.11.1, Theorem 2.11.4]{nova-mono}, the author has investigated the estimate \eqref{isti-ee} and case $p(x)\equiv p \in [1,\infty),$ where the resulting function $G(\cdot)$ is also bounded and continuous (see also \cite{fedorov-novi} and \cite{podgorica}). We would like to note that Theorem \ref{jensen} provides a new way of looking at the invariance of the (equi-)Weyl-$p$-almost periodicity under the action of infinite convolution product 
 as well as that the (equi-)Weyl-$p$-almost periodicity in \cite[Theorem 2.11.4]{nova-mono} can be proved directly from Corollary \ref{kraj}. Let us explain this in more detail. 
Let a function $g : {\mathbb R} \rightarrow X$ be (equi-)Weyl-$p$-almost periodic. Then the function
$G : {\mathbb R} \rightarrow Y,$ defined through \eqref{wer}, is (equi-)Weyl-$p$-almost periodic and we can show this in the following way. 
It is clear that the function $\check{g}(\cdot)$ is also (equi-)Weyl-$p$-almost periodic. By Corollary \ref{kraj}, with an arbitrary sequence of positive real numbers such that $\sum_{k=0}^{\infty}a_{k}=1$ and the function $\varphi(x)\equiv x$, observing also that the class of (equi-)Weyl-$p$-almost periodic functions is closed under pointwise multiplications with scalars, it suffices to show, by considering the function $(M^{-1}R(t))_{t>0}$ for a sufficiently large real number $M>0,$ that for every real numbers $t\in {\mathbb R}$ and $l>0$ we have
\begin{align}\label{bluz}
 \int^{t+l}_{t}\Biggl( \sum_{k=0}^{\infty}\Bigl( \int^{(k+1)l}_{kl}\frac{t^{(\beta -1)q}\, dt}{(1+t^{\gamma})^{q}}\Bigr)^{1/q} \Biggr)^{p}\, dx \leq \mbox{Const.},
\end{align}
provided that $p>1,$ resp.
\begin{align}\label{bluz1}
 \int^{t+l}_{t}\sum_{k=0}^{\infty}\Bigl\| \frac{\cdot^{\beta -1}}{1+\cdot^{\gamma}}\Bigr\|_{L^{\infty}[kl,(k+1)l]}\, dx \leq \mbox{Const.},
\end{align}
provided that $p=1.$ As 
$$
\int^{(k+1)l}_{kl}\frac{t^{(\beta -1)q}\, dt}{(1+t^{\gamma})^{q}} \leq \frac{1}{1+k^{q\gamma}l^{q\gamma}}(k+1)^{(\beta-1)q}l^{(\beta-1)q+1},\quad k\in {\mathbb N}_{0},
$$
the estimate \eqref{bluz} follows from the inequality $(\beta-1+(1/q)-\gamma)p+1\leq 0,$ which is true. The estimate \eqref{bluz1} is much simpler and follows from the inequality $\gamma>1.$

Concerning Theorem \ref{jensen-12121121} and Theorem \ref{jensen-121211212},
we will provide two illustrative examples:

\begin{example}\label{ne-mogu}
Suppose that $\phi(x)=x^{\alpha},$ $x\geq 0,$ where $\alpha>0.$ 
If the estimate \eqref{isti-ee} holds, then condition \eqref{radio121} holds provided that, for every $x\in {\mathbb R}$ and $l>0,$ we have
$$
\sum_{k=0}^{\infty}k^{\beta-1-\gamma}\bigl[ F(l,-x+kl )\bigr]^{(-1)/\alpha}<\infty,
$$
while condition \eqref{radio121123} holds provided that, for every $t\in {\mathbb R}$ and $l>0,$ we have
$$
\int^{t+l}_{t}
\Biggl( \Biggl(\frac{1}{1+k^{q\gamma}l^{q\gamma}}(k+1)^{(\beta-1)q}l^{(\beta-1)q+1}\Biggr)^{1/q}\Biggl(\frac{F(l,t)}{F(l,-x+kl)}\Biggr)^{1/\alpha}\Biggr)^{p}\, dx \leq 1,
$$
if $p>1,$ resp.
$$
\int^{t+l}_{t}
\frac{(kl)^{\beta -1}}{1+k^{\gamma}l^{\gamma}}  \Biggl(\frac{F(l,t)}{F(l,-x+kl)}\Biggr)^{1/\alpha}\, dx \leq 1,
$$
if $p=1.$ If the estimate
\eqref{rajasad} holds, then condition \eqref{radio121} holds provided that, for every $x\in {\mathbb R}$ and $l>0,$ we have
$$
\sum_{k=0}^{\infty}e^{-ck}k^{\beta-1}\bigl[ F(l,-x+kl )\bigr]^{(-1)/\alpha}<\infty,
$$
while condition \eqref{radio121123} holds provided that, for every $t\in {\mathbb R}$ and $l>0,$ we have
$$
\int^{t+l}_{t}
\Biggl( e^{-ck}(kl)^{\beta -1}\Biggl(\frac{F(l,t)}{F(l,-x+kl)}\Biggr)^{1/\alpha}\Biggr)^{p}\, dx \leq 1.
$$
\end{example}

\begin{example}\label{ne-mogu-jog}
Suppose that condition (A2) holds. If the estimate \eqref{isti-ee} holds, then condition
\eqref{radio121okee} holds provided that
$$
\sum_{k=0}^{\infty}k^{\beta-1-\gamma}F(l,-x+kl)^{-1}<\infty ,
$$
while condition \eqref{radio1211212345} holds provided that, for every $t\in {\mathbb R}$ and $l>0,$ we have
$$
\int^{t+l}_{t}
\Biggl( \Biggl(\frac{1}{1+k^{q\gamma}l^{q\gamma}}(k+1)^{(\beta-1)q}l^{(\beta-1)q+1}\Biggr)^{1/q}\frac{F(l,t)}{F(l,-x+kl)}\Biggr)^{p}\, dx \leq 1,
$$
if $p>1,$ resp.
$$
\int^{t+l}_{t}
\frac{(kl)^{\beta -1}}{1+k^{\gamma}l^{\gamma}}  \frac{F(l,t)}{F(l,-x+kl)}\, dx \leq 1,
$$
if $p=1.$ If the estimate
\eqref{rajasad} holds, then \eqref{radio121okee} holds provided that, for every $x\in {\mathbb R}$ and $l>0,$ we have
$$
\sum_{k=0}^{\infty}e^{-ck}k^{\beta-1}\bigl[ F(l,-x+kl )\bigr]^{(-1)}<\infty,
$$
while
condition \eqref{radio1211212345} holds provided that, for every $t\in {\mathbb R}$ and $l>0,$ we have
$$
\int^{t+l}_{t}\Biggl( e^{-ck}(kl)^{\beta -1}\frac{F(l,t)}{F(l,-x+kl)}\Biggr)^{p}\, dx \leq 1.
$$
\end{example}

\subsection{Some applications}\label{nekiprimeri}

In this subsection, we will briefly explain how one can provide certain applications of obtained theoretical results to abstract degenerate fractional differential inclusions.
For further information about fractional calculus and fractional differential equations, we refer to 
\cite{bajlekova}, \cite{Diet}, \cite{kilbas}, \cite{knjigaho}; for abstract degenerate integro-differential equations and inclusions, we refer to \cite{faviniyagi}, \cite{FKP} and \cite{svir-fedorov}.

We need to remind ourselves of the following definitions of Caputo fractional derivatives and Weyl-Liouville fractional derivatives.
Suppose $\zeta \in (0,1)$ and $u :[0,\infty) \rightarrow X$ satisfies, for every $T>0,$ $u \in C((0,T]: X),$ $u(\cdot)-u(0) \in L^{1}((0,T) : X)$
and $g_{1-\zeta}\ast (u(\cdot)-u(0)) \in W^{1,1}((0,T) : X).$ Then the Caputo fractional derivative $
{\mathbf
D}_{t}^{\zeta}u(t)$ is defined by
$$
{\mathbf
D}_{t}^{\zeta}u(t):=\frac{d}{dt}\Biggl[g_{1-\zeta}
\ast \Bigl(u(\cdot)-u(0)\Bigr)\Biggr](t),\quad t\in (0,T].
$$
The Weyl-Liouville fractional derivative
$D_{t,+}^{\zeta}u(t)$ of order $\zeta$ is defined for those continuous functions
$u : {\mathbb R} \rightarrow X$
such that $t\mapsto \int_{-\infty}^{t}g_{1-\zeta}(t-s)u(s)\, ds,$ $t\in {\mathbb R}$ is a well-defined continuously differentiable function, by
$$
D_{t,+}^{\zeta}u(t):=\frac{d}{dt}\int_{-\infty}^{t}g_{1-\zeta}(t-s)u(s)\, ds,\quad t\in {\mathbb R}.
$$
Set $
{\mathbf D}_{t}^{1}u(t):=(d/dt)u(t)$ and $
D_{t,+}^{1}u(t):=-(d/dt)u(t).$

Of importance is the following abstract Cauchy relaxation differential inclusion
\begin{align}\label{left-bruka}
D_{t,+}^{\zeta}u(t)\in -{\mathcal A}u(t)+g(t),\ t\in {\mathbb R},
\end{align}
where $0<\zeta \leq 1,$ the function  $g: {\mathbb R} \rightarrow X$ satisfies certain assumptions and ${\mathcal A}$ is a closed MLO. If ${\mathcal A}$ satisfies condition \cite[(P)]{nova-mono}, then there exists a strongly
continuous operator family $(R_{\zeta}(t))_{t>0}$ satisfying the esimate of type \eqref{isti-ee}, in the case $\zeta \in (0,1),$ or estimate of type \eqref{rajasad}, in the case $\zeta=1.$ Since a mild solution of \eqref{left-bruka}
is given by $t\mapsto u(t)\equiv \int^{t}_{-\infty}R_{\zeta}(t-s)g(s),$ $t\in {\mathbb R},$ applications of results obtained in this section are straightforward.

Of importance is also the following fractional relaxation inclusion
\[
\hbox{(DFP)}_{f,\zeta} : \left\{
\begin{array}{l}
{\mathbf D}_{t}^{\zeta}u(t)\in {\mathcal A}u(t)+f(t),\ t> 0,\\
\quad u(0)=u_{0},
\end{array}
\right.
\]
where $0<\zeta \leq 1,$ $u_{0} \in X,$ the function  $f: [0,\infty) \rightarrow X$ satisfies certain assumptions and ${\mathcal A}$ is a closed MLO. If ${\mathcal A}$ satisfies condition \cite[(P)]{nova-mono}, then
there exists a strongly
continuous operator family $(S_{\zeta}(t))_{t>0}$ satisfying the esimate of type \eqref{isti-ee}, in the case $\zeta \in (0,1),$ or estimate of type \eqref{rajasad}, in the case $\zeta=1,$ such that the unique mild solution of problem $\hbox{(DFP)}_{f,\zeta}$
is given by $t\mapsto u(t)\equiv S_{\zeta}(t)u_{0}+ \int^{t}_{0}S_{\zeta}(t-s)f(s),$ $t\geq 0,$
where $u_{0}$ belongs to the continuity set of $(S_{\zeta}(t))_{t>0},$ i.e., $\lim_{t\rightarrow 0+}S_{\zeta}(t)u_{0}=u_{0}.$ Moreover, $\lim_{t\rightarrow +\infty}S_{\zeta}(t)u_{0}=0$
and Proposition \ref{finite} can be straightforwardly applied.

\end{document}